\documentclass[12pt,a4paper,reqno]{amsart}

\usepackage[usenames,dvipsnames]{color}

\usepackage{tikz,graphicx}
\usepackage[all,arc,curve,color,frame,graph,matrix,cmtip,poly]{xy}
\usepackage{tikz-cd}
\usetikzlibrary{decorations.markings}
\tikzstyle directed=[postaction={decorate,decoration={markings,
    mark=at position .65 with {\arrow{latex}}}}]

\usepackage{amsfonts,amsmath,amssymb,color,amscd,amsthm}
\usepackage[T1]{fontenc}

\usepackage{enumerate}

\makeatletter
\renewcommand\p@enumii{}
\makeatother

\usepackage{ulem}
\usepackage{comment}

\usepackage{cases}
\usepackage{units}

\usepackage[]{hyperref}
\hypersetup{
    colorlinks=true,
    allcolors=violet,
    linktocpage=true
}
\usepackage[capitalise,noabbrev]{cleveref}

\newtheorem{maintheorem}{Theorem}

\newtheorem{lemma}{Lemma}
\newtheorem{theorem}[lemma]{Theorem}
\newtheorem{definition}[lemma]{Definition}
\newtheorem{proposition}[lemma]{Proposition}
\newtheorem{corollary}[lemma]{Corollary}

\newtheorem{remark}[lemma]{Remark}

\theoremstyle{example}

\theoremstyle{idea}

\newtheorem*{question*}{Question}
\newtheorem*{problem*}{Problem}

\usepackage{stackengine}

\newcommand\A{{\mathbb A}}
\newcommand\C{{\mathbb C}}

\newcommand\F{{\mathbb F}}
\newcommand\FFin{{\mathbf F}}
\renewcommand\H{{\mathbb H}}

\newcommand\p{{\mathbb P}}
\newcommand\Q{{\mathbb Q}}
\newcommand\R{{\mathbb R}}

\newcommand\Z{{\mathbb Z}}

\DeclareMathOperator{\Sym}{\mathcal S}

\newcommand{\smat}[4]{\left( \begin{smallmatrix} #1\,&#2\\ #3\,&#4 \end{smallmatrix} \right)}

\renewcommand{\Sym}{\mathfrak S}
\newcommand{\Alt}{\mathfrak A}
\newcommand\Bir{{\mathrm{Bir}}}

\newcommand\Aff{{\mathrm{Aff}}}
\newcommand\Bl{{\mathrm{Bl}}}

\newcommand\GL{{\mathrm{GL}}}

\newcommand\PGL{{\mathrm{PGL}}}
\newcommand\SL{{\mathrm{SL}}}

\DeclareMathOperator{\sgn}{sgn}

\DeclareMathOperator{\PSL}{PSL}

\DeclareMathOperator{\Fix}{Fix}

\newcommand{\Jonq}{{\rm Jonq}}
\newcommand{\JJ}{\Jonq_{p_0}}

\newcommand\Aut{{\mathrm{Aut}}}

\renewcommand{\phi}{{\varphi}}

\DeclareMathOperator{\Stab}{Stab}

\DeclareMathOperator{\Pic}{Pic}

\DeclareMathOperator{\rk}{rk}

\DeclareMathOperator{\pr}{pr}
\DeclareMathOperator{\Cent}{Cent}

\DeclareMathOperator{\length}{\ell}
\DeclareMathOperator{\Ax}{Ax}

\newcommand{\bp}{{\rm Bir}({\mathbb P}^2)}

\usepackage{mathtools}

\title[The derived length of subgroups of the Cremona group]{The derived length of subgroups\\ of the Cremona group}

\author{Jean-Philippe Furter}
\address{Institut de Math\'ematiques de Bordeaux, Universit\'e de Bordeaux, CNRS, UMR 5251, France}
\email{jean-philippe.furter@math.u-bordeaux.fr}

\author{Isac Hed\'en}
\address{Department of Mathematics, Uppsala University, Sweden}
\email{isac.heden@math.uu.se}

\begin{document}
\maketitle

\begin{abstract}
We prove that the maximal derived length of a solvable subgroup of the plane Cremona group over $\C$ is equal to $5$. The main difficulty is the case of subgroups containing loxodromic elements. To treat this case, we prove that the maximal derived length of a finite solvable subgroup is $4$ and establish a centraliser theorem: the centraliser of a finite solvable subgroup of derived length $4$ contains no loxodromic element. The proof of the centraliser theorem relies on Blanc's classification of maximal algebraic subgroups and then uses equivariant Sarkisov theory, orbit estimates, superrigidity, and fixed curves of genus at least two.
\end{abstract}

\tableofcontents
\section{Introduction}

For a group $G$, set $D^0G:=G$, and $D^{i+1}G:=[D^iG,D^iG]$. If $G$ is solvable, its \textit{derived length}, denoted by $\ell(G)$, is the smallest non-negative integer $r$ such that $D^rG=1$.

For an arbitrary group $\Gamma$, we set
\begin{align*}
\psi(\Gamma)& :=\sup\bigl\{\ell(G)\mid G \subseteq \Gamma \text{ is solvable}\bigr\},\\
\psi_{\mathrm{finite}}(\Gamma)&:=
\sup\bigl\{\ell(G)\mid G\subseteq\Gamma
\text{ is finite and solvable}\bigr\},
\end{align*}
and for a subgroup $\Gamma\subseteq\bp$, we set 
\begin{align*}
\psi_{\mathrm{bounded}}(\Gamma)&:=
\sup\bigl\{\ell(G)\mid G\subseteq\Gamma
\text{ is bounded and solvable}\bigr\}.
\end{align*}
Recall that a subset $A \subseteq \bp$ is called bounded if there exists a constant $K$ such that $\deg(g)\leq K$ for all $g \in A$.

Derived length has long been studied for linear groups. Newman computed  the exact value $\psi \bigl( \textrm{GL}(n,\C) \bigr)$ for all $n$ (see \cite[Theorem A, p.~60]{Newman1972}). It turns out that $\psi \bigl( \textrm{GL}(n,\C) \bigr)$ is equivalent to $5\log_9(n)$ as $n$ goes to infinity (see \cite[Theorem 3.10]{Wehrfritz1973}). Let us give a few particular values for $\psi \bigl( \textrm{GL}(n,\C) \bigr)$.

\[
\begin{tabular}{|c|l|l|l|l|l|l|l|l|l|l|l|l|l|l|l|}
\hline
$n$  & $1$ & $2$ & $3$ & $4$ & $5$  & $6$ & $7$ & $8$ & $9$ & $10$  \\
\hline
$\psi \bigl( \textrm{GL}(n,\C) \bigr) $  & $1$  & $4$ & $5$ & $6$ & $7$ & $7$ & $7$ & $8$ & $9$ & $10$  \\
\hline
\end{tabular}
\]
\medskip

He has  also computed the values of $\psi ( \Sym_n ) $ for all $n$ in \cite[Theorem~B, p.~61]{Newman1972}. Here are some particular values:

\[
\begin{tabular}{|c|l|l|l|l|l|l|l|l|l|l|l|l|l|l|l|}
\hline
$n$  & $1$ & $2$ & $3$ & $4$ & $5$  & $6$ & $7$ & $8$ & $9$ & $10$  \\
\hline
$\psi ( \Sym_n ) $  & $0$  & $1$ & $2$ & $3$ & $3$ & $3$ & $3$ & $4$ & $5$ & $5$  \\
\hline
\end{tabular}
\]

If $\operatorname{Diff}(\C^n,0)$ denotes the group of germs of local complex
analytic diffeomorphisms at the origin, then we have
\[ \psi\bigl(\operatorname{Diff}(\C,0)\bigr)=2 \]
(see \cite{Loray1999} and \cite{IlyashenkoYakovenko2008}), and
Ribón proved in \cite{Rib2019} that we have
\[ \psi\bigl(\operatorname{Diff}(\C^n,0)\bigr)=2n+1
\qquad\text{for } 2 \leq n \leq 5. \]
Let $\operatorname{Aff}_n(\C)$ be the group of affine automorphisms of the complex affine $n$-space $\A^n$ and let $\Aut (\A^n)$  be the group of  polynomial automorphisms. By \cite[Propositions 3.14 \& 3.16]{FurterPoloni2018}, we have
\[ \psi \bigl( \operatorname{Aff}_2(\C)  \bigr)  =  \psi \bigl( \Aut (\A^2) \bigr) = 5.\]
More recently, Regeta, Urech and van Santen proved that the maximal derived length of connected solvable subgroups of the birational transformation group of an $n$-dimensional complex variety is at most $2n$ with equality if and only if the variety is birationally equivalent to $\p^n$ \cite[Theorem~1.3]{RegetaUrechVanSanten2025}. In particular, the maximal derived length of connected solvable subgroups of the plane Cremona group is equal to~$4$. This last result is also \cite[Lemma 6.5]{FurterHeden2023}.

For arbitrary solvable subgroups, Urech proved the upper bound
\[ \psi\bigl(\bp)\leq8 \]
\cite[Theorem~1.10]{Urech2021} and since $\operatorname{Aff}_2 (\C) \subseteq \bp$ satisfies $ \psi \bigl( \operatorname{Aff}_2(\C)  \bigr)  =5$, one obtains
\[ 5 \leq \psi\bigl(\bp) \leq 8. \]
Our main result closes this gap:

\begin{maintheorem} \label{theorem: main}
Every solvable subgroup of $\bp$ has derived length at most five. Moreover, this bound is attained. Equivalently, we have $ \psi \bigl( \bp \bigr)=5$.
\end{maintheorem}

\begin{remark}
Since $\Bir (\p^1) = \Aut (\p^1) = \PGL_2 (\C)$,
we have $ \psi \bigl( \Bir (\p^1) \bigr)=3$ by Lemma~\ref{lemma: psi(PGL_2(C))=3} below.
\end{remark}

The main difficulty in proving Theorem~\ref{theorem: main} is the case of a solvable subgroup containing a loxodromic element. We treat this case using the following two results.

\begin{maintheorem}\label{theorem: psi_{finite}(Bir(P2)) =4}
Every finite solvable subgroup of $\bp$ has derived length at most four, and this bound is attained. Equivalently, we have $\psi_{\mathrm{finite}}(\bp)=4$.
\end{maintheorem}

The second result provides the main geometric ingredient in the proof.

\begin{maintheorem}[Centraliser theorem]\label{theorem: the centraliser of a solvable group of derived length 4 has no loxodromic element}
Let $F \subseteq \bp$ be a finite solvable group of derived length $4$. Then $ \Cent_{\bp}(F)$ contains no loxodromic element.
\end{maintheorem}

We now explain how these two results yield the bound $\ell(H)\leq 5$ when $H\subseteq\bp$ is a solvable subgroup containing a loxodromic element. Recall that $\bp$ acts faithfully by isometries on the Picard--Manin hyperbolic space $\H^\infty$; see
\cite[\S3]{Cantat2011}.
The induced isometries are classified according to their dynamical type:  elliptic, parabolic or
loxodromic. A loxodromic element $f$ has exactly two fixed points on $\partial \H^\infty$, and the geodesic line joining these two points is called its axis, denoted $\Ax(f)$. In Section~\ref{subsection: Solvable groups with loxodromic elements} we show that if $H\subseteq\bp$ is a solvable subgroup containing a loxodromic element $f$ then $H$ preserves $\Ax(f)$ as a set. Hence there is a homomorphism
\[ H \longrightarrow \mathrm{Isom}(\Ax(f)) \simeq \mathrm{Isom}(\R). \]
Since $\mathrm{Isom}(\R)$ has derived length two, it follows that $D^2H$ lies in the pointwise stabilizer of the axis. If this stabilizer is infinite, a theorem of Cantat shows that $H$ is conjugate to a subgroup of the
monomial group  $\Aut \bigl( (\C^*)^2 \bigr) =(\C^*)^2 \rtimes \GL_2(\Z)$ from which we obtain $\length(H)\leq 5$.

If it is finite,
Theorem~\ref{theorem: psi_{finite}(Bir(P2)) =4} gives the preliminary estimate
\[ \ell(H) \leq 6. \]
Equality would force $D^2H$ to be a finite solvable subgroup of
derived length four. A suitable power of the original loxodromic
element would then centralize $D^2H$, contradicting
Theorem~\ref{theorem: the centraliser of a solvable group of derived length 4 has no loxodromic element}. Thus the centraliser theorem excludes precisely the extremal case left open by the preceding estimate, and lowers the bound from six to five.

We prove Theorem~\ref{theorem: the centraliser of a solvable group of derived length 4 has no loxodromic element} using Blanc's classification of maximal algebraic subgroups of the plane Cremona group, see Theorem~\ref{theorem: the eleven families of maximal algebraic subgroups of Bir(P2)} below. This reduces the problem to eleven families of del Pezzo surfaces and conic bundles, which we treat using the tools developed in Section~\ref{section: geometric tools}. The eleven cases are handled by three main geometric mechanisms. 

First, for the primitive case on $\p^2$, and for $\p^1\times\p^1$, orbit bounds imply equivariant birational superrigidity,  while for the Fermat cubic we use Blanc's superrigidity result directly. Second, the equivariant Sarkisov program is used to force preservation of a rational fibration for intransitive linear actions on $\p^2$, for Hirzebruch surfaces, and in the exceptional conic-bundle analysis. Third, in the remaining situations, an involution in $F$ fixes a curve of genus at least two, which is necessarily preserved by the centraliser and excludes loxodromic transformations.

For the proof of Theorem~\ref{theorem: main}, it now remains to consider solvable subgroups containing no loxodromic element. Such a group either contains a parabolic element or consists entirely of elliptic elements.

In the first case, the group is conjugate either to a subgroup of the Jonquières group or to a subgroup of the automorphism group of a Halphen surface. The required bound then follows from Proposition~\ref{proposition: psi(Jonq)=5} or Proposition~\ref{proposition: psi(Aut(Halphen)) leq 5}, respectively. In the second case, we use Urech's description of subgroups consisting
entirely of elliptic elements \cite[Theorem~1.4]{Urech2021}: such a group either preserves a rational fibration, is bounded, or is a torsion group. We will in particular use the sharp bound $\psi_{\mathrm{bounded}}(\bp)=5$ (see \cite[Lemma~7.3]{Urech2021}).

The article is organized as follows. In
Section~\ref{section: finite bounded solvable subgroups}, we prove
Theorem~\ref{theorem: psi_{finite}(Bir(P2)) =4}. We also give, for
completeness, a short proof of the above mentioned result $\psi_{\mathrm{bounded}}(\bp)=5$ (this is our Proposition~\ref{proposition: psi_{bounded}(Bir(P2)) =5}). Section~\ref{section: geometric tools} recalls and develops the geometric tools used in the proof of the centraliser theorem, which is carried out in Section~\ref{sec:centraliser-proof}.
Section~\ref{sec:dynamical-bounds} applies the centraliser theorem to prove the derived-length bound for solvable subgroups containing a loxodromic element. It also establishes the bounds for the Jonqui\`eres group and for automorphism groups of Halphen surfaces. Finally, Section~\ref{section: proof of the main theorem} assembles these ingredients to complete the proof of Theorem~\ref{theorem: main}.

\medskip
\noindent\textbf{Acknowledgements.} The authors gratefully acknowledge financial support from G.~S.\ Magnuson's Fund (grant MG2023-0148), which supported the first author's research visit to Uppsala, and from the Esseen Travel Scholarship, which supported the second author's research visit to Bordeaux.

\medskip
\noindent\textbf{AI declaration.} The research project and substantial parts of the mathematical work
were developed long before the authors began using AI. During the
subsequent development of the proofs and preparation of the manuscript,
the authors used ChatGPT (OpenAI; GPT-5.6 Sol Pro).

Its mathematical assistance included suggestions concerning the use of the equivariant
Sarkisov program and helped explore possible approaches in several cases of Blanc's classification of maximal algebraic subgroups, notably exceptional conic bundles and $(\mathbb{Z}/2\mathbb{Z})^2$-conic bundles.

The authors worked through these suggestions, and developed and presented the resulting arguments
in their own way. The overall mathematical development, organisation,
and exposition of the article were carried out by the authors.
ChatGPT was also used as a writing aid, mainly to suggest
wording for individual sentences and short passages, which the authors
subsequently reviewed and revised. All mathematical claims and
references were checked by the authors, who take full responsibility
for the results and the final manuscript.

\section{Finite and bounded solvable subgroups}\label{section: finite bounded solvable subgroups}

The goal of this section is to prove Theorem~\ref{theorem: psi_{finite}(Bir(P2)) =4} and to recall the sharp bound for bounded solvable subgroups in Proposition~\ref{proposition: psi_{bounded}(Bir(P2)) =5}. We first establish some elementary facts about $\PGL_2(\C)$ and $\GL_2(\C)$, and recall Blanc's classification of maximal algebraic subgroups of $\bp$.

Throughout this article, for a solvable group $G$, we denote its derived length by $\ell(G)$. We will repeatedly use the following elementary inequality: if $1\to N\to G\to Q\to 1$ is a short exact sequence of solvable groups, then $\ell(G)\leq \ell(N)+\ell(Q)$. Indeed, if $\ell(Q)=r$, then $D^rG\subseteq N$, and the inequality follows. Analogously, if $1\to N\to G\to Q\to 1$ is any short exact sequence of groups, we also have $\psi (G)\leq \psi(N)+\psi(Q)$.

\subsection{Subgroups of $\PGL_2(\C)$ and $\GL_2(\C)$}

\begin{lemma} \label{lemma: finite subgroups of PGL2 fixing a point are cyclic}
Let $H$ be a finite subgroup of $\Aut(\p^1)$ fixing a point. Then $H$ is cyclic.
\end{lemma}

\begin{proof}
Let $p$ be this fixed point. The tangent representation
\[ H \longrightarrow \GL(T_p\p^1)\simeq\C^* \]
is faithful, since a finite-order automorphism fixing $p$ and acting
trivially on $T_p\p^1$ is the identity. Hence $H$ is cyclic.
\end{proof}

\begin{lemma} \label{lemma: size orbits for the actions of A4, S4, and A5 on P1}
Let $G \subseteq \Aut(\p^1)$ be isomorphic to $\Alt_4$, $\Sym_4$, or $\Alt_5$. Then the cardinality of every $G$-orbit belongs respectively to
\[ \{4,6,12\},\qquad \{6,8,12,24\},\qquad \{12,20,30,60\}.\]
\end{lemma}

\begin{proof}
Let $G\subseteq\Aut(\p^1)$ be a finite subgroup. For any point $p\in\p^1$, its $G$-orbit has $|G|/|G_p|$ elements, where $G_p$ denotes the stabilizer of $p$. By Lemma~\ref{lemma: finite subgroups of PGL2 fixing a point are cyclic}, $G_p$ is cyclic. The result now follows since the cyclic subgroups of $\Alt_4$ (resp.\ $\Sym_4$, resp.\ $\Alt_5$) have cardinality $1,2$ or $3$ (resp. $1,2,3$ or $4$, resp. $1,2,3$ or $5$).
\end{proof}

\begin{remark} All these orbit lengths are attained. Since we will not use this fact, we omit the straightforward verification. \end{remark}

The following elementary fact will be used repeatedly throughout the article.

\begin{lemma} \label{lemma: psi(PGL_2(C))=3}
We have $\psi ( \PGL_2 (\C ) ) =3$. Moreover a solvable subgroup of $\PGL_2 (\C)$ has derived length $3$ if and only if it is isomorphic to $\Sym_4$.
\end{lemma}

\begin{proof}
For any solvable subgroup of $\PGL_2 (\C)$, its Zariski closure is solvable with the same derived length. Moreover, up to conjugation, each algebraic subgroup $H$ of  $\PGL_2 (\C)$ satisfies one of the following assertions (see e.g. \cite[Theorem 1]{Nguyen-vanDerPut-Top2008}):
\begin{enumerate}[$($i$)$]
\item \label{the whole group}
$H = \PGL_2 (\C)$;
\item \label{subgroup of the upper-triangular matrices}
$H$ is a subgroup of the upper-triangular matrices;
\item \label{diagonal and antidiagonal matrices}
$H = \{  \smat{a }{0}{0}{b}, \; a,b \in \C^* \} \cup \{  \smat{0 }{b}{a}{0}, \; a,b \in \C^* \}$;
\item \label{a finite group}
$H$ is either equal to $D_{2n}$ (the dihedral group of order $2n$), or $\Alt_4$ (the tetrahedral group), or $\Sym_4$ (the octahedral group), or $\Alt_5$ (the icosahedral group).
\end{enumerate}
In case \eqref{the whole group}, the group is not solvable. In cases \eqref{subgroup of the upper-triangular matrices} and \eqref{diagonal and antidiagonal matrices}, it is solvable of derived length at most $2$. In case \eqref{a finite group}: the group $\Alt_5$ is not solvable, the groups $D_{2n}$ and $\Alt_4$ are solvable of derived lengths at most $2$. Finally, the group $\Sym_4$ is solvable of derived length~$3$. The conclusion follows.
\end{proof}

\begin{remark} \label{remark: C(y) embedding in C}
The same statement also holds with $\C$ replaced by $\C(y)$. Indeed, there exists an abstract field embedding
$\C(y) \hookrightarrow \C$: an algebraic closure of $\C(y)$ and $\C$ are algebraically closed fields of characteristic zero with the same transcendence degree over $\Q$, and hence are isomorphic. Thus $\PGL_2(\C(y))\hookrightarrow\PGL_2(\C)$, so every solvable subgroup of $\PGL_2(\C(y))$ has derived length at most $3$, with equality if and only if it is isomorphic to $\Sym_4$.
\end{remark}

The following result is classical (see e.g.~\cite{Beauville2010}).

\begin{lemma} \label{lemma: finite subgroups of PGL_2(C)}
Any finite subgroup of~$\PGL_2 (\C)$ is isomorphic to one of $\Z /n \Z$, $D_{2n}$ $($the dihedral group of order $2n$$)$, $\Alt_4$, $\Sym_4$, or $\Alt_5$, and there is only one conjugacy class for each of these groups.
\end{lemma}

\begin{lemma} \label{lemma: A solvable subgroup of derived length 4 of GL_2(C) contains the subgroup {I,-I}}
A solvable subgroup of derived length $4$ of $\GL_2 (\C)$ contains the subgroup $\{ \pm I \}$, where $I$ denotes the identity matrix of $\GL_2 (\C)$.
\end{lemma}

\begin{proof}
Let $G$ be this subgroup. Let $\pi \colon \GL_2 (\C) \to \PGL_2 (\C)$ be the canonical surjection. Since the kernel of $G \to \pi (G)$ consists of scalars, it is abelian, and hence we get $\length ( \pi (G) ) \ge 4-1 = 3$. By Lemma~\ref{lemma: psi(PGL_2(C))=3}, we have $ \pi (G) = \Sym_4$. Since the determinant of a commutator is $1$, the derived subgroup $G'$ of $G$ lies in $\SL_2 (\C)$. Let $\tau \colon \SL_2 (\C) \to \PSL_2 (\C) = \PGL_2 (\C)$ be the canonical surjection. We have $\tau (G') =\pi (G) ' = (\Sym_4)' = \Alt_4$. The kernel of $G' \to \Alt_4$ lies in $\{ \pm I \}$. If it were trivial, then we would have $G' \simeq \Alt_4$, hence $\length (G') = 2$, hence $\length (G) = 3$. A contradiction. Hence the kernel of $G' \to \Alt_4$ is $\{ \pm I \}$ and this concludes the proof.
\end{proof}

\subsection{Maximal algebraic subgroups of the Cremona group}

We shall use the following two results of Blanc, which will play a major role in our arguments. They can be found in \cite[Theorems~1 and~2]{Blanc2009}. We first recall two definitions from \cite[\S2.4 and \S2.5]{Blanc2009} that will be needed to state Theorem~\ref{theorem: the eleven families of maximal algebraic subgroups of Bir(P2)}.

\begin{definition} \label{definition: exceptional conic bundle}
A conic bundle $(S, \pi)$ with $2n$ singular fibres is said to be exceptional if it admits at least two sections with self-intersection $-n$ for some $n \ge 1$. It is then a fact that the number of these sections is exactly two, and that these two sections are disjoint. Moreover, for each singular fibre $E_1 + E_2$, the component $E_1$ meets one of the sections, and the component $E_2$ the other one. These properties are for example stated in \cite[Lemma 4.3.1]{Blanc2009}.
\end{definition}

\begin{definition} \label{definition: (Z/2Z)^2-conic bundle}
A conic bundle $\pi \colon S \to \p^1$ is called a
$(\Z /2 \Z)^2$-conic bundle if
\[ \Aut(S/ \p^1)\simeq (\Z / 2 \Z)^2 \]
and each of its three nontrivial involutions $\sigma$ fixes pointwise an irreducible curve $C_{\sigma}$ which is a double cover of $\p^1$, via $\pi$, branched over a nonempty set $A_{\sigma}$ of even cardinality.
\end{definition}

\begin{theorem}[Blanc] \label{theorem: each algebraic subgroup is contained in a maximal algebraic subgroup}
Every algebraic subgroup of the Cremona group is contained in a maximal algebraic subgroup.
\end{theorem}

\begin{theorem}[Blanc's classification] \label{theorem: the eleven families of maximal algebraic subgroups of Bir(P2)}

Up to conjugacy in $\bp$, the maximal algebraic subgroups are
of the following eleven types:
\begin{enumerate}
\item[(1)] $\Aut(\p^2)$;

\item[(2)] $\Aut(\p^1\times\p^1)$;

\item[(3)] $\Aut(S_6)$, where $S_6$ is the del Pezzo surface of degree $6$;

\item[(4)] $\Aut(\F_n)$, where $\F_n$ is the $n$-th Hirzebruch surface and $n \geq 2$;

\item[(5)] $\Aut(S,\pi)$, where $(S,\pi)$ is an exceptional conic bundle
with at least four singular fibres;

\item[(6)] $\Aut(S_5)$, where $S_5$ is the del Pezzo surface of degree $5$;

\item[(7)] $\Aut(S)$, where $S$ is a del Pezzo surface of degree $4$;

\item[(8)] $\Aut(S)$, where $S$ is either a cyclic cubic surface, the
Clebsch cubic surface, or belongs to a one-parameter family of smooth
cubic surfaces whose automorphism groups are isomorphic to $\Sym_4$;

\item[(9)] $\Aut(S)$, where $S$ is a del Pezzo surface of degree $2$,
a double cover of $\p^2$ branched over a smooth quartic $Q_S$, such
that $\Aut(Q_S)$ acts without fixed points on the complement of the
bitangency points;

\item[(10)] $\Aut(S)$, where $S$ is a del Pezzo surface of degree $1$;

\item[(11)] $\Aut(S,\pi)$, where $(S,\pi)$ is a
$(\Z/2\Z)^2$-conic bundle and $S$ is not a del Pezzo surface.
\end{enumerate}
\end{theorem}

\begin{remark}\label{remark: which automorphism groups from the theorem of J. Blanc are finite}
The algebraic groups in
Theorem~\ref{theorem: the eleven families of maximal algebraic subgroups of Bir(P2)}
are infinite in cases~(1)--(5) and finite in cases~(6)--(11).
This follows directly from the descriptions in Blanc's classification.
For cases~(6)--(10), one may also use the classical fact that the
automorphism group of a smooth complex del Pezzo surface of degree $d$
is finite if and only if $d\leq 5$; see for instance
\cite[p.~2]{DolMar2024}.
In case~(11), finiteness follows from the structure of
$\Aut(S,\pi)$ recalled in Case~(11) below.
\end{remark}

\subsection{Finite solvable subgroups}

The goal of this section is to prove
Theorem~\ref{theorem: psi_{finite}(Bir(P2)) =4}.
Blanc's classification reduces the finite case to the eleven automorphism groups listed in Theorem~\ref{theorem: the eleven families of maximal algebraic subgroups of Bir(P2)}. For the first of these, $\Aut(\p^2)=\PGL_3(\C)$, we will use the classical classification of finite subgroups of $\PGL_3(\C)$ due to Blichfeldt \cite{Blichfeldt1917}. We begin by recalling the terminology that we need.

A finite subgroup $F\subseteq\PGL_3(\C)$ is called
\textit{intransitive} if it fixes a point of $\p^2$, and
\textit{transitive} otherwise. A transitive subgroup is called
\textit{imprimitive} if it preserves a triangle, i.e.\ a set of three
non-collinear points of $\p^2$, and \textit{primitive} otherwise.

If $F$ is intransitive and fixes $p\in\p^2$, let $\widetilde F\subseteq \SL_3(\C)$ be the inverse image of $F$ by the canonical surjection $\SL_3(\C) \to \PSL_3 (\C) = \PGL_3 (\C)$.
The point $p$ corresponds to a one-dimensional $\widetilde F$-invariant subspace of $\C^3$. Since $\widetilde F$ is finite, complete reducibility gives a $\widetilde F$-invariant complement, and hence an $F$-invariant line $D \subseteq \p^2$ not containing $p$. Therefore, after conjugation, we may assume that
\[ p=[1:0:0], \qquad D =\{x=0\}, \]
and identify $F$ with a finite subgroup $G\subseteq\GL_2(\C)$ via
\[ A \longmapsto \bigl([x:y:z]\longmapsto[x:A(y,z)]\bigr). \]

\begin{proof}[Proof of Theorem~\ref{theorem: psi_{finite}(Bir(P2)) =4}]
Let $F\subseteq\bp$ be a finite solvable subgroup. By Theorems~\ref{theorem: each algebraic subgroup is contained in a maximal algebraic subgroup} and \ref{theorem: the eleven families of maximal algebraic subgroups of Bir(P2)},  after birational conjugation we may assume that $F$ is contained in one of the eleven maximal algebraic groups listed there. 

\smallskip

\noindent\textbf{Case~(1): $\Aut(\p^2)=\PGL_3(\C)$.}

If $F$ is intransitive, the discussion above identifies $F$, after
conjugation, with a finite subgroup of $\GL_2(\C)$. Hence
\[ \ell(F)\leq\psi(\GL_2(\C))=4. \]

Suppose next that $F$ is transitive and imprimitive. Then $F$ preserves
a triangle, and there is a short exact sequence
\[
1\longrightarrow D\longrightarrow F\longrightarrow H\longrightarrow1,
\]
where $D$ is contained in the diagonal torus $(\C^*)^2$ and
$H\subseteq\Sym_3$. Thus
\[ \ell(F) \leq  \psi \bigl( (\C^*)^2 ) + \psi ( \Sym_3) = 1+2=3.
\]

It remains to consider the primitive case. By the classification of
finite primitive subgroups of $\PGL_3(\C)$, the group $F$ is isomorphic
to one of
\[ \Alt_5, \qquad \Alt_6, \qquad \PSL_2(\mathbb \FFin_7),\qquad H_{216}, \]
or to one of two proper subgroups of $H_{216}$; see
\cite[Theorem~4.8, p.~465]{DolgachevIskovskikh2009}. The first three groups are not solvable, so it remains to consider subgroups of the Hessian group
\[ H_{216} =V\rtimes L, \qquad V=(\FFin_3)^2, \qquad L= \SL_2(\FFin_3)\simeq 2T, \] where $L$ acts naturally and irreducibly on $V$ and $2T$ denotes the binary tetrahedral group; see \cite[Proposition~4.1]{ArtebaniDolgachev2009}.

We will prove the following result.

\begin{lemma} \label{lemma: length-four subgroup of H_{216}}
We have $\length( H_{216} )=4$. Moreover, if $F \subseteq H_{216} $ is a subgroup of derived length $4$, then $F=H_{216}$.
\end{lemma}

\begin{proof}
1) We begin by proving that $\length ( H_{216} ) =4$. Recall that $L \simeq 2T = Q_8 \rtimes \Z / 3 \Z$, where $Q_8$ is the Quaternion group, and that the derived series of $L$ is
\[ L  \; \triangleright \;  Q_8  \;  \triangleright \;  \{ \pm I \}  \;  \triangleright  \; 1. \]
Let us prove that the derived series of $H_{216} =  V \rtimes L$ is
\begin{equation} H_{216} = V \rtimes L \; \triangleright \;  V \rtimes Q_8 \; \triangleright \;  V \rtimes  \{ \pm I \}  \;  \triangleright \;  V \;   \triangleright \;  1. \label{equation: derived series of V rtimes L} \end{equation}
Since $V$ is abelian, for every subgroup $H \subseteq L$, we have
\[ (V \rtimes H)' = [V,H] \rtimes H'\]
where $G'$ denotes the derived subgroup of a group $G$ and where $[V, H]$ denotes the subgroup of $V$ generated by the elements of the form $h(v) -v$, $h \in H$, $v \in V$. Therefore, for showing that \eqref{equation: derived series of V rtimes L} is the derived series of $V \rtimes L$ it is enough to check that we have $[V, L] = [V, Q_8] = [V, \{ \pm I \} ] = V$. Since we have $[V, \{ \pm I \} ]  \subseteq [V, Q_8] \subseteq [V, L] \subseteq V$, it is enough to check that we have $[V, \{ \pm I \}] =V$. The central element $-I$ belongs to $Q_8 \subseteq L = \SL_2 (\FFin_3)$ and acts on $V$ by multiplication by $-1$. Hence, for every $v \in V$,
\[ (-I) (v) - v = -2 v = v,\]
showing that we have $[V, \{ \pm I \}] =V$. This shows that \eqref{equation: derived series of V rtimes L} is the derived series of $V \rtimes L$. Hence we have actually proven that $\ell(H_{216})=4$.

2) We finally prove that if $F \subseteq H_{216}$ satisfies $\ell(F)=4$, then we have $F = H_{216}$. Let $K\subseteq L$ be the projection of $F$. Since $F \cap V$ is abelian, we have $\ell(K)=3$.
This forces $K=L$: indeed, the short exact sequence
\[ 1 \longrightarrow \{ \pm I \} \longrightarrow L \xrightarrow{\hspace{1.5mm} \pi \hspace{1.5mm}  } \Alt_4  \longrightarrow 1 \]
induces the short exact sequence
\[ 1 \longrightarrow \{ \pm I \} \cap K \longrightarrow K \longrightarrow  \pi (K)   \longrightarrow 1 \]
from which we get $\pi (K) = \Alt_4$, $\{ \pm I \} \cap K = \{ \pm I \}$, and finally $K=L$.

Thus $F \cap V$ is an $L$-submodule of the irreducible
module $V$. If $F \cap V=0$, then $F \simeq L$ has derived length $3$, a contradiction. Hence $F \cap V=V$, and therefore $F = H_{216}$.
\end{proof}

We summarize the conclusion of Case~(1) in the following lemma.

\begin{lemma} \label{lemma: finite solvable subgroups of Aut P2}
Let $F\subseteq \Aut(\p^2)$ be a finite solvable subgroup. Then we have
$\ell(F)\leq 4$, and  this bound is attained. Moreover, if
$\ell(F)=4$, then exactly one of the following occurs:
\begin{enumerate}
\item[(1)] $F$ is intransitive, i.e.\ under conjugation
in $\Aut(\p^2)$, it is the image of a finite subgroup of $\GL_2 (\C)$ under
\[ A \mapsto  [ x : A(y,z)] .\]
\item[(2)] Up to conjugation in $\Aut(\p^2)$, we have $F=H_{216}$.
\end{enumerate}
\end{lemma}

\noindent\textbf{Case~(2): $\Aut(\p^1\times\p^1)$.} We have
$\Aut(\p^1\times\p^1)
 =\bigl(\PGL_2(\C)\times\PGL_2(\C)\bigr)\rtimes\Z/2\Z$, where the nontrivial element of $\Z/2\Z$ exchanges the two factors. Hence we have
\[ \psi  \bigl(   \Aut(\p^1\times\p^1)   \bigr)  \le \psi \Big(  \bigl( \PGL_2 ( \C) \bigr)^2  \Big) + \psi ( \Z / 2 \Z) = 3 + 1 = 4,\] 
and thus $\ell(F) \leq 4$.

\smallskip
\noindent\textbf{Case~(3): $\Aut(S_6)$.}
Blanc's description gives
\[ \Aut(S_6)\simeq (\C^*)^2 \rtimes (\Sym_3 \times \Z/2\Z). \]
The first factor is abelian and the second has derived length $2$.
Thus $\ell(F)\leq3$.

\smallskip
\noindent \textbf{Case~(4): $\Aut(\F_n)$, $n\geq2$.} We have
\[ \Aut(\F_n)\simeq \C^{n+1} \rtimes \bigl( \GL_2(\C)/\mu_n \bigr). \]
Since $\C^{n+1}$ has no nontrivial finite subgroup, $F$ embeds in $\GL_2(\C)/\mu_n$. Its inverse image in $\GL_2(\C)$ is solvable, and therefore $\ell(F) \leq \psi \bigl( \GL_2(\C) \bigr) =4$.

\smallskip
\noindent \textbf{Case~(5): exceptional conic bundles.} Let $\Delta \subseteq \p^1$ be the set of points over which $\pi$ has a
singular fibre, and let $H_\Delta \subseteq \PGL_2(\C)$ be the subgroup of automorphisms leaving $\Delta$ invariant. If $s_1,s_2$ are the two exceptional sections, the actions on the base and on the set $\{s_1,s_2\}$ define homomorphisms
\[ \rho\colon\Aut(S,\pi)\longrightarrow H_\Delta, \qquad \sigma\colon\Aut(S,\pi)\longrightarrow\Z/2\Z. \]
By \cite[Lemmas~4.3.1 and~4.3.3]{Blanc2009}, the map
$(\rho, \sigma)$ is surjective and its kernel is a one-dimensional
torus. Thus there is a short exact sequence
\begin{equation} \label{equation: rho and sigma}
1 \longrightarrow T \longrightarrow \Aut(S,\pi) \xrightarrow{(\rho,\sigma)} H_\Delta \times \Z/2\Z \longrightarrow 1, \qquad T \simeq\C^*. \end{equation}
This gives
\[ \ell(F) \le \psi \bigl( \Aut(S,\pi) \bigr) \le \psi (T) + \psi \bigl( H_\Delta \times \Z/2\Z \bigr) \le  \psi (T) + \psi \bigl(  \PGL_2 (\C)  \times \Z/2\Z \bigr) =4. \]

\smallskip
\noindent\textbf{Case~(6): $\Aut(S_5)$.} We have $\Aut(S_5) \simeq \Sym_5$, and therefore
$\ell(F)\leq \psi(\Sym_5)=3$; cf. the table given in the introduction.

\smallskip
\noindent\textbf{Case~(7): del Pezzo surfaces of degree $4$.}
Such a surface is the blow-up of a set $\Delta\subseteq\p^2$ of five points in general position (i.e.\ such that no three of them are collinear), and Blanc's description gives
\[ \Aut(S)\simeq (\Z/2\Z)^4 \rtimes H_S, \]
where $H_S$ is the subgroup of $\PGL_3(\C)$ preserving $\Delta$.
The five points lie on a unique smooth conic, so $H_S$ embeds in $\Aut(\p^1) = \PGL_2(\C)$ and preserves a set of five points on $\p^1$. By Lemma~\ref{lemma: size orbits for the actions of A4, S4, and A5 on P1}, the groups $\Alt_4$, $\Sym_4$, and $\Alt_5$ cannot occur. Hence $H_S$ is cyclic or dihedral  (see Lemma~\ref{lemma: finite subgroups of PGL_2(C)}), and therefore $\ell(H_S)\leq2$. It follows that
\[ \ell(F) \leq \ell \bigl( \Aut(S) \bigr) \leq1+2=3. \]
\smallskip
\noindent\textbf{Case~(8): cubic surfaces.}
We use Blanc's description of the automorphism groups occurring in this family, distinguishing the three types of cubic surfaces listed in Theorem~\ref{theorem: the eleven families of maximal algebraic subgroups of Bir(P2)}.

If $S$ is the Fermat cubic, then
\[ \Aut(S)\simeq(\Z/3\Z)^3\rtimes\Sym_4, \]
and hence every solvable subgroup has derived length at most $4$.

Suppose next that $S$ is a cyclic cubic surface. If $S$ is not the
Fermat cubic, Blanc gives a short exact sequence
\[ 1 \longrightarrow\Z/3\Z\longrightarrow\Aut(S) \longrightarrow H_\Gamma \longrightarrow 1, \]
where $H_\Gamma$ embeds in the automorphism group of a smooth elliptic curve $\Gamma$. Since $\Aut(\Gamma)$ has derived length at most $2$, we obtain $\ell(\Aut(S))\leq3$.

Finally, the Clebsch cubic has automorphism group $\Sym_5$, while the remaining cubic surfaces in family~(8) have automorphism group $\Sym_4$. Thus $\ell(F)\leq 3$ in these last two cases, and hence $\ell(F)\leq 4$ in all cases.

\smallskip
\noindent\textbf{Case~(9): del Pezzo surfaces of degree $2$.}
Let $Q_S\subseteq\p^2$ be the smooth quartic branch curve of the anticanonical double cover $S \to \p^2$. Blanc's description gives
\[ \Aut(S) \simeq \Z/2\Z \times H_S, \qquad H_S=\Aut(\p^2,Q_S)\subseteq\PGL_3(\C). \]
We claim that every solvable subgroup of $H_S$ has derived length at most $3$. Otherwise, let $H\subseteq H_S$ have derived length $4$. By Case~(1), either $H$ is intransitive or $H\simeq H_{216}$.

The latter is impossible: the curve $Q_S$ has genus $3$, so the Hurwitz bound gives
\[ |\Aut(Q_S)|\leq84(3-1)=168<216. \]
Suppose therefore that $H$ is intransitive. After conjugation, $H$ is the image of a finite subgroup $G \subseteq \GL_2(\C)$ under
\[ A \longmapsto [x : A (y,z)]. \]
The image of $G$ in $\PGL_2(\C)$ has derived length $3$, and is therefore isomorphic to $\Sym_4$ by Lemma~\ref{lemma: psi(PGL_2(C))=3}.
The line $D =\{x=0\}$ is $H$-invariant, so $Q_S \cap D$ is an $H$-invariant effective divisor of degree $4$ on $D \simeq \p^1$. A contradiction, because every $\Sym_4$-orbit on $\p^1$ has cardinality at least $6$ by Lemma~\ref{lemma: size orbits for the actions of A4, S4, and A5 on P1}.

Thus every solvable subgroup of $H_S$ has derived length at most $3$. Since $\Aut(S)\simeq\Z/2\Z\times H_S$, it follows that
$\ell(F)\leq3$.
\smallskip

\noindent\textbf{Case~(10): del Pezzo surfaces of degree $1$.} The bi-anticanonical morphism maps $S$ two-to-one onto a quadric cone. Let $p\in S$ be the unique point lying over its vertex. Equivalently, $p$ is the unique zero-dimensional component of the fixed locus of the Bertini involution; see \cite[p.~12]{BayleBeauville2000}. Since the Bertini involution is central, every automorphism of $S$ fixes $p$.

The tangent representation
\[ \Aut(S)\longrightarrow\GL(T_pS) \simeq \GL_2(\C) \]
is faithful: indeed, a finite-order automorphism fixing $p$ and acting trivially on $T_pS$ is the identity. Hence $\ell(F) \leq \psi \bigl( \GL_2 (\C) \bigr) =4$.

\smallskip
\noindent\textbf{Case~(11): $(\Z/2\Z)^2$-conic bundles.} The action on the base gives a short exact sequence
\begin{equation} \label{equation: short exact sequence for Aut(S,pi) for a Z/2Z-conic bundle}
1\longrightarrow V \longrightarrow\Aut(S,\pi) \xrightarrow{\rho}H_V \longrightarrow1, \qquad V \simeq (\Z/2\Z)^2, \quad H_V \subseteq \PGL_2(\C). \end{equation}
By Lemma~\ref{lemma: psi(PGL_2(C))=3}, we obtain $\ell(F)\leq4$.

We have proved that every finite solvable subgroup of $\bp$ has derived length at most $4$, and  that this bound is attained (for example by $H_{216}$).
\end{proof}

\begin{remark}\label{remark: remark on cases 3,6,7,9}
Note that the proof moreover shows that a finite solvable subgroup $F\subseteq\bp$ with $\ell(F)=4$ cannot occur in cases~(3), (6), (7) and~(9).
\end{remark}

\subsection{Bounded solvable subgroups}

The upper bound in the following proposition was established by Urech \cite[Lemma~7.3]{Urech2021}. We include a short proof for completeness, and also to make explicit a minor adjustment needed in the cubic-surface case of the proof in loc.\ cit.

\begin{proposition} \label{proposition: psi_{bounded}(Bir(P2)) =5}
We have $ \psi_{\mathrm{bounded}}(\bp)=5$.
\end{proposition}

\begin{proof}
The lower bound follows from
$\psi(\Aff_2(\C))=5$ \cite[Proposition~3.16]{FurterPoloni2018},
since $\Aff_2(\C)\subseteq\bp$ is bounded.

For the upper bound, every bounded subgroup of $\bp$ is contained in an algebraic subgroup \cite[Corollary~2.18]{BlancFurter2013}, hence in one of the maximal algebraic subgroups of Theorem~\ref{theorem: the eleven families of maximal algebraic subgroups of Bir(P2)}. In cases (2), (3), and (5), the estimates in the proof of Theorem~\ref{theorem: psi_{finite}(Bir(P2)) =4} apply verbatim to arbitrary solvable subgroups; in cases (6)--(11), the maximal algebraic group is
finite, so the same theorem gives the bound $4$. Finally,
\[ \psi(\PGL_3(\C)) \leq \psi(\GL_3(\C))=5 \]
and
\[ \psi(\Aut(\F_n)) \leq 1+\psi(\GL_2(\C)/\mu_n) \leq 1+\psi(\GL_2(\C))=5. \]
Thus $\psi_{\mathrm{bounded}}(\bp) \leq 5$, and the result follows.
\end{proof}

\section{Geometric tools for the centraliser theorem}
\label{section: geometric tools}

The proof of the centraliser theorem uses three main mechanisms to exclude loxodromic elements: superrigidity, preservation of a rational fibration, and preservation of a curve of genus at least two. We collect in this section the geometric tools needed for these arguments.

\subsection{Criteria for excluding loxodromic elements}

The following result is well known.

\begin{lemma}[Rational fibrations] \label{lemma: rational fibration}
Let $f\in\bp$.  If $f$ preserves a rational fibration, then $f$ is not loxodromic.
\end{lemma}

\begin{lemma}[Invariant curves of high genus] \label{lemma: preserved  curve of genus at least 2 implies not loxodromic}
Let $X$ be a smooth projective rational surface and let $f \colon X \dashrightarrow X$ be birational.  Suppose that $f$ preserves an irreducible curve whose normalization has genus at least two.  Then $f$ is not loxodromic.
\end{lemma}

\begin{proof}
By \cite[Theorem~0.1]{DillerFavre2001}, after a birational modification we may assume that $f$ is algebraically stable. A curve of positive genus cannot be contracted by a birational morphism between smooth surfaces, so the strict transform of the given curve remains invariant and has the same geometric genus.  Diller--Jackson--Sommese prove that an invariant connected curve for an algebraically stable birational map of dynamical degree $>1$ has arithmetic genus zero or one \cite[Theorem~1.1]{DJS07}.  This is impossible for an irreducible curve of geometric genus at least two.
\end{proof}

\begin{corollary}[Central involutions with a high-genus fixed curve]\label{corollary: central involution with a high-genus fixed curve}
Let $\sigma\in\Aut(X)$ be an involution of a smooth projective rational surface.  Assume that
$\Fix(\sigma)$ has a unique irreducible curve component $C$ of positive genus and
that $g(C)\geq2$. If $f\in\Bir(X)$ commutes with $\sigma$, then $f$ is not loxodromic.
\end{corollary}

\begin{proof}
For a general point $x\in C$ at which $f$ is defined,
\[ \sigma(f(x))=f(\sigma(x))=f(x). \]
The curve $C$ cannot be contracted by $f$, because every curve contracted by a birational map between smooth surfaces is rational.  Hence the closure of $f(C)$ is a curve contained in $\Fix(\sigma)$.  The induced map from the normalization of
$C$ to the normalization of its image is birational, so the image has the same positive genus as $C$.  By uniqueness it is~$C$.  Apply Lemma~\ref{lemma: preserved  curve of genus at least 2 implies not loxodromic}.
\end{proof}

\subsection{Equivariant Sarkisov theory and superrigidity}
We recall the part of the two-dimensional equivariant Sarkisov program that will be used below.  Our main references are \cite{Iskovskikh1996}, \cite[\S7]{DolgachevIskovskikh2009}, \cite[\S2]{CheltsovTschinkelZhang2025}  and \cite[Chapter~16]{LamyCremonaBook}. Unless explicitly stated otherwise, throughout this subsection $G$ is a finite group acting biregularly on a smooth projective rational surface. 

If $X$ and $X'$ are $G$-surfaces, a birational map
$\varphi\colon X\dashrightarrow X'$ is called $G$-equivariant if
\[ \varphi\circ g=g\circ\varphi \qquad\text{for every }g\in G. \]
In particular, every element of $\Cent_{\Bir(X)}(G)$ is $G$-equivariant.

\begin{definition}\label{definition: G-Mori fibre space}
A \textit{$G$-Mori fibre space} on a smooth projective rational surface is a $G$-equivariant morphism $\pi\colon S\to B$ of one of the following two forms:
\begin{enumerate}[$($i$)$]
\item $B$ is a point, $S$ is a del Pezzo surface, and
      $\operatorname{rk}\Pic(S)^G=1$;
\item $B\simeq\p^1$, the morphism $\pi$ is a conic bundle, and
      $\operatorname{rk}\Pic(S)^G=2$.
\end{enumerate}
In the first case we call $S$ a minimal $G$-del Pezzo surface, and in the second we call $(S,\pi)$ a minimal $G$-conic bundle.
\end{definition}

By the two-dimensional equivariant Sarkisov program, every
$G$-equivariant birational map between $G$-Mori fibre spaces factors as a composition of $G$-equivariant isomorphisms and $G$-Sarkisov links; see \cite[Theorem~16.28(1)]{LamyCremonaBook}. We recall the four types of links in Figure~\ref{fig:sarkisov-links}; see also \cite[p.~276]{Blanc2009}.

\begin{figure}[ht]
\centering

\begin{minipage}[t]{0.21\textwidth}
\centering
\textbf{Type I}
\[
\begin{tikzcd}[row sep=1.7em,column sep=1.8em]
S \arrow[d,"\pi"'] &
S' \arrow[l,"\sigma"'] \arrow[d,"\pi'"] \\
\mathrm{pt} & \p^1 \arrow[l]
\end{tikzcd}
\]
\end{minipage}
\hfill
\begin{minipage}[t]{0.28\textwidth}
\centering
\textbf{Type II}
\[
\begin{tikzcd}[row sep=1.7em,column sep=1.5em]
S \arrow[d,"\pi"'] &
X \arrow[l,"\sigma"'] \arrow[r,"\tau"] &
S' \arrow[d,"\pi'"] \\
Y & &
Y' \arrow[ll,"\sim"']
\end{tikzcd}
\]
\end{minipage}
\hfill
\begin{minipage}[t]{0.21\textwidth}
\centering
\textbf{Type III}
\[
\begin{tikzcd}[row sep=1.7em,column sep=1.8em]
S \arrow[r,"\sigma"] \arrow[d,"\pi"'] &
S' \arrow[d,"\pi'"] \\
\p^1 \arrow[r] &
\mathrm{pt}
\end{tikzcd}
\]
\end{minipage}
\hfill
\begin{minipage}[t]{0.21\textwidth}
\centering
\textbf{Type IV}
\[
\begin{tikzcd}[row sep=1.7em,column sep=1.8em]
S \arrow[r,"\sim"] \arrow[d,"\pi"'] &
S' \arrow[d,"\pi'"] \\
\p^1 &
\p^1
\end{tikzcd}
\]
\end{minipage}

\caption{The four types of $G$-Sarkisov links between $G$-Mori fibre spaces.}
\label{fig:sarkisov-links}
\end{figure}
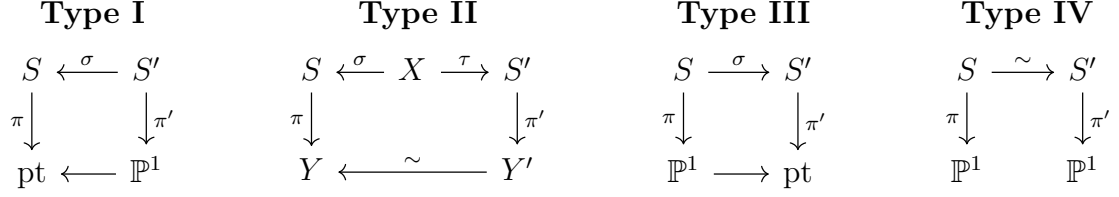

In Type~I one passes from a minimal $G$-del Pezzo surface to a
minimal $G$-conic bundle, and Type~III is the inverse operation.
A link of Type~II joins two minimal $G$-del Pezzo surfaces or two
minimal $G$-conic bundles; in the latter case the bases are identified. Here $\sigma$ and $\tau$ are $G$-equivariant birational morphisms contracting $G$-orbits. A link of Type~IV identifies the underlying surfaces but changes the $G$-equivariant conic-bundle structure.

In particular, a Type~II link between conic bundles is an elementary transformation: $\sigma$ blows up a $G$-orbit of points lying on pairwise distinct smooth fibres, and then $\tau$ contracts the strict transforms of these fibres. 

A Sarkisov factorization is called reduced if no two consecutive links are inverse to one another, up to $G$-equivariant isomorphisms of Mori fibre spaces.

For conic bundles we will also use the following standard terminology (see e.g.\ \cite[Definition~16.10]{LamyCremonaBook}).

\begin{definition} \label{definition: G-twisted and untwisted singular fibre} Let $\pi \colon S \to \p^1$ be a $G$-conic bundle. A singular fibre $E_1+E_2$ is called \textit{$G$-twisted} if there exists $g\in G$ such that
\[ g(E_1)=E_2. \]
Otherwise it is called \textit{$G$-untwisted}.
\end{definition}

The following characterization of minimality is classical.

\begin{lemma} \label{lemma: characterization of minimal G-conic bundles via twistedness}
A $G$-conic bundle $\pi\colon S\to\p^1$ is $G$-minimal if and only if all its singular fibres are $G$-twisted.
\end{lemma}

\begin{definition} \label{definition: superrigid G-Mori fibre space}
$($see e.g.\ \cite[Proposition 7.1.2, p.~275]{Blanc2009}, \cite[Definition~7.10, p.~530]{DolgachevIskovskikh2009}, \cite[\S2, p.~10]{CheltsovTschinkelZhang2025}$)$
A $G$-Mori fibre space $\pi \colon S \to B$ is called superrigid if for any $G$-Mori fibre space $\pi' \colon S' \to B'$ and any birational $G$-map $\varphi \colon S \dasharrow S'$, the map $\varphi$ is an isomorphism of $G$-Mori fibre spaces. Equivalently, there are no $G$-Sarkisov links starting from the $G$-Mori fibre space $\pi \colon S \to B$.
\end{definition}

\begin{lemma}[Standard orbit bound] \label{lemma: orbit bound for del Pezzo links}
Let $S$ be a minimal $G$-del Pezzo surface, and suppose that a $G$-Sarkisov link starting from $S$ begins by blowing up
a $G$-orbit $\Omega\subseteq S$. Then
\[ |\Omega|<K_S^2. \]
\end{lemma}

\begin{proof}
By the classification of two-dimensional equivariant Sarkisov links,
the blow-up
\[ \widehat S\longrightarrow S \]
of the centre $\Omega$ is again a del Pezzo surface; see \cite[\S2]{CheltsovTschinkelZhang2025}. Hence
\[ 0<K_{\widehat S}^2=K_S^2-|\Omega|, \]
and the result follows.
\end{proof}

In particular, since $K_{\p^2}^2=9$, the centre of a first $G$-Sarkisov link from $\p^2$ has cardinality less
than $9$. Similarly, since $K_{\p^1\times\p^1}^2=8$, the centre
of a first $G$-Sarkisov link from the del Pezzo $G$-surface $\p^1\times\p^1$ has cardinality less than $8$.

The following classical result (see for example \cite[Corollary 7.11]{DolgachevIskovskikh2009}) is a direct consequence of Lemma~\ref{lemma: orbit bound for del Pezzo links}.

\begin{lemma} \label{lemma: a minimal del Pezzo G-surface of degree d with no orbit of size <d is G-superrigid}
Let $S$ be a minimal del Pezzo $G$-surface of degree $d =K_S^2$. If $S$ has no $G$-orbit of size $< d$, then $S$ is $G$-superrigid.
\end{lemma}

\begin{proof}
Otherwise a reduced Sarkisov factorization of a non-isomorphic $G$-equivariant birational map starting from $S$ would have a first link, contradicting Lemma~\ref{lemma: orbit bound for del Pezzo links}.
\end{proof}

\subsection{Equivariant birational maps of Hirzebruch surfaces}

We now record two consequences of the preceding Sarkisov description. The first controls the possible passage from $\p^2$ to $\F_1$, and the second applies it to birational self-maps of odd Hirzebruch surfaces.

\begin{lemma}[Unique first link]
\label{lemma: unique first link in the intransitive linear case}
Let $F\subseteq\PGL_3(\C)$ be a finite, intransitive, solvable subgroup of derived length $4$, and let $p\in\p^2$ be a fixed point of $F$. Then $p$ is the unique fixed point of $F$, and the only possible first $F$-Sarkisov link from $\p^2$ is the link obtained by blowing up $p$, namely
\[ \pi^{-1}\colon \p^2\dashrightarrow \F_1=\Bl_p(\p^2), \]
where $\pi\colon\F_1\to\p^2$ is the blow-up morphism.
\end{lemma}

\begin{proof}
Using the notation above, after conjugation we may assume that
\[ p=[1:0:0], \qquad D= \{x=0\}, \]
and identify $F$ with a finite subgroup $G\subseteq\GL_2(\C)$ via
\[ A \longmapsto\bigl( [x:y:z] \mapsto [x:A(y,z)] \bigr). \]
Let
\[ \rho \colon \GL_2(\C) \longrightarrow \PGL_2(\C) \]
be the canonical projection and put $P:=\rho(G)$. Since
$\ker(\rho|_G)$ consists of scalar matrices, it is abelian. Hence
$\length(P)\geq3$, and Lemma~\ref{lemma: psi(PGL_2(C))=3} gives
\[ P \simeq \Sym_4. \]
Moreover, by Lemma~\ref{lemma: A solvable subgroup of derived length 4 of GL_2(C) contains the subgroup {I,-I}}, we have $-I\in G$.

Projection from $p$ identifies $D$ with the base of the pencil of
lines through $p$, and the induced action on this base is precisely $P \simeq \Sym_4$. In particular, by Lemma~\ref{lemma: size orbits for the actions of A4, S4, and A5 on P1}, every $F$-orbit contained in $D$ has cardinality at least $6$.

Now let
\[ q=[1:v] \in \p^2 \setminus(D \cup \{p\}), \qquad 0 \neq v \in \C^2. \]
The natural map
\[ F \cdot q \longrightarrow P \cdot [v] \]
is surjective. Since $-I\in G$, every fibre contains the two distinct points $[1:Av]$ and $[1:-Av]$. Therefore
\[ | F\cdot q | \;  \geq  \; 2 \, | P\cdot[v] |  \; \geq  \; 12. \]
Thus no point of $D$ or of $\p^2 \setminus(D \cup \{p\})$ is fixed by $F$, and hence $p$ is the unique fixed point of $F$.

Since $\{p\}$ and $D$ are $F$-invariant, every $F$-orbit is
contained in one of
\[ \{p\},\qquad D,\qquad \p^2 \setminus (D \cup \{p\}). \]

Let $\Omega\subseteq \p^2$ be the centre of a first $F$-Sarkisov link. By Lemma~\ref{lemma: orbit bound for del Pezzo links},
\[ | \Omega | <9. \]
The preceding estimate excludes $\Omega \subseteq \p^2 \setminus(  D \cup \{p\} )$.

Suppose that $\Omega\subseteq D$. Then $|\Omega| \geq 6$. Let
\[ \eta \colon \widehat S\longrightarrow \p^2 \]
be the blow-up of $\Omega$, and let $\widehat D$ be the strict
transform of $D$. Since $\widehat S$ is a del Pezzo surface,
$-K_{\widehat S}$ is ample, whereas
\[ (-K_{\widehat S}) \cdot \widehat D =3-|\Omega|<0, \]
a contradiction. Hence $\Omega$ is not contained in $D$.

The only remaining possibility is $\Omega= \{p \}$. Therefore the first link is
\[ \pi^{-1}\colon\p^2 \dashrightarrow \F_1.\qedhere \]
\end{proof}

We now use Lemma~\ref{lemma: unique first link in the intransitive linear case} to prove the following proposition.

\begin{proposition}[Equivariant self-maps of odd Hirzebruch surfaces]
\label{proposition: equivariant self-maps of odd Hirzebruch surfaces}
Let $n\geq 1$ be odd, let
\[ \pi_n \colon \F_n \longrightarrow \p^1 \]
be the standard ruling, and let $F \subseteq\Aut(\F_n)$ be a finite
solvable subgroup of derived length $4$. Then every
$F$-equivariant birational self-map
\[ \varphi\colon \F_n\dashrightarrow\F_n \]
preserves the ruling $\pi_n$. In particular, $\varphi$ is not
loxodromic.
\end{proposition}

\begin{proof}
1) Let us begin by showing that the image of $F$ in $\Aut(\p^1)$ is isomorphic to $\Sym_4$. We have $\Aut(\F_n)\simeq
\C^{n+1}\rtimes\bigl(\GL_2(\C)/\mu_n\bigr)$.
Since $\{0\}$ is the only finite subgroup of $\C^{n+1}$, the projection $F\longrightarrow\GL_2(\C)/\mu_n$ is injective. The natural action on the base fits into an exact sequence
\[ 1 \longrightarrow \C^*/\mu_n \longrightarrow \GL_2(\C)/\mu_n
\longrightarrow \PGL_2(\C) \longrightarrow 1. \]
Thus the kernel of the induced action of $F$ on the base is abelian. Since $F$ has derived length $4$, its image on the base has derived length at least $3$. By Lemma~\ref{lemma: psi(PGL_2(C))=3}, it has derived length exactly $3$ and is therefore isomorphic to $\Sym_4$.

2) Factor $\varphi$ into a reduced sequence of $F$-equivariant Sarkisov links and isomorphisms of Mori fibre spaces.

Consider first a type-II link between Hirzebruch surfaces. Such a link is an elementary transformation whose centre is an $F$-orbit
$\Omega$ consisting of points lying on pairwise distinct fibres. Hence projection to the base induces a bijection from $\Omega$ onto an orbit of the image of $F$ in $\Aut(\p^1)$. Since this image is isomorphic to $\Sym_4 $, Lemma~\ref{lemma: size orbits for the actions of A4, S4, and A5 on P1} shows that $|\Omega|$ is even.

An elementary transformation at one point changes the Hirzebruch index by one, up to sign. Consequently, a type-II link centred at $\Omega$ preserves the parity of the Hirzebruch index. Since $n$ is odd, every Hirzebruch surface occurring in the factorization has odd index. In particular, the factorization never reaches $\F_0$. Therefore no link of type~IV can occur.

We next exclude links of type~III. Among Hirzebruch surfaces, such a link can only be the contraction
\[ \F_1\longrightarrow \p^2 \]
of the $(-1)$-section. Suppose that such a link occurs. It cannot be the last link of the factorization, since the target of $\varphi$ is $\F_n$. The induced action of $F$ on the resulting $\p^2$ is faithful and still has derived length $4$. Moreover, the point obtained by contracting the $(-1)$-section is fixed by $F$. Hence Lemma~\ref{lemma: unique first link in the intransitive linear case} shows that the next link must be the blow-up of this point, that is, the inverse type-I link
\[ \p^2\dashrightarrow\F_1. \]
Thus two consecutive links are inverse to each other, up to the intervening isomorphisms of Mori fibre spaces. This contradicts the reducedness of the factorization. Hence no link of type~III occurs.

Since the factorization starts from a conic bundle, a link of type~I could occur only after a link of type~III had first reached a del Pezzo surface. Hence the absence of type-III links also excludes type-I links. We have therefore shown that every link in the factorization is of type~II. All these links are elementary transformations over the same base, and consequently $\varphi$ preserves the ruling $\pi_n$.

The last assertion follows from Lemma~\ref{lemma: rational fibration}.
\end{proof}

\section{The centraliser theorem}
\label{sec:centraliser-proof}

Let $F \subseteq \Bir(\p^2)$ be a finite solvable subgroup with
$\ell(F)=4$. Blanc's classification allows us to assume that $F$ is
contained in one of the eleven maximal algebraic groups listed there.

The proof of Theorem~\ref{theorem: psi_{finite}(Bir(P2)) =4} shows that the families (3), (6), (7), and (9) cannot contain a finite solvable subgroup of derived length $4$. It therefore remains to treat the families (1), (2), (4), (5), (8), (10), and (11). We treat these families below in the order of Blanc's classification.

\subsection{The projective plane}
By Lemma~\ref{lemma: finite solvable subgroups of Aut P2}, there are two possibilities: either $F$ is intransitive (i.e. admits at least one fixed point in $\p^2$), or, up to conjugation, we have $F=H_{216}$.

\vspace{2mm}

\noindent \emph{The intransitive case.}
Let $p \in\p^2$ be a fixed point of $F$, and let
\[ \pi \colon \F_1=\Bl_p(\p^2)\longrightarrow\p^2 \]
be the blow-up of $p$.

Let $\varphi\in\Cent_{\bp}(F)$. Since $p$ is $F$-fixed, the action of $F$ lifts to $\F_1$, and
\[ \widetilde\varphi :=\pi^{-1}\circ\varphi\circ\pi \colon\F_1\dashrightarrow\F_1 \]
is $F$-equivariant. By Proposition~\ref{proposition: equivariant self-maps of odd Hirzebruch surfaces}, applied with $n=1$, the map $\widetilde\varphi$ preserves the ruling of $\F_1$. Hence $\varphi$ preserves the pencil of lines through $p$, and therefore is not loxodromic by Lemma~\ref{lemma: rational fibration}.

\vspace{2mm}

\noindent\emph{The primitive case.}
Up to conjugation, we have $F=H_{216}$. Since every $H_{216}$-orbit in $\p^2$ has cardinality at least~$9$ (see \cite[Proposition~3.17]{Sakovics2019}), Lemma~\ref{lemma: a minimal del Pezzo G-surface of degree d with no orbit of size <d is G-superrigid} shows that $\p^2$ is $H_{216}$-superrigid. It follows that $\Cent_{\bp}(H_{216})$ is contained in $\Aut(\p^2)$ and in particular it contains no loxodromic element.

\subsection{The product of two projective lines}

We now consider family~(2) of Theorem~\ref{theorem: the eleven families of maximal algebraic subgroups of Bir(P2)}, and begin by determining the subgroups of derived length $4$ in this family.

\begin{lemma} \label{lemma: length-four subgroups of Aut(P1xP1)}
Let $\tau=(y,x)$ be the standard involution of $\p^1\times\p^1$. Up to conjugation, there is a unique solvable subgroup of derived length $4$ of
\[ \Aut(\p^1\times\p^1) =\bigl(\PGL_2(\C)\times\PGL_2(\C)\bigr) \rtimes \langle \tau \rangle, \]
namely the group
\[ (\Sym_4\times\Sym_4) \rtimes \langle \tau \rangle. \]
of order $24^2 \times 2 = 1152$.
\end{lemma}

\begin{proof}
Let $F \subseteq \Aut (\p^1 \times \p^1)$ be solvable of derived length $4$. Put
\[ B := \PGL_2(\C) \times \PGL_2(\C),\]
and set  $N:=F\cap B$. Every solvable subgroup of $B$ has derived length at most $3$, by
Lemma~\ref{lemma: psi(PGL_2(C))=3} applied to the two projections. Hence $F\not\subseteq B$, and therefore
\[ F/N\simeq\Z/2\Z. \]
Since $F'\subseteq N$ and $\length(F)=4$, it follows that
\[ \length(N)=3. \]

Let
\[ P_i:=\operatorname{pr}_i(N)\subseteq\PGL_2(\C), \qquad i=1,2. \]
Choose an element
\[ t=(g,h)\tau\in F\setminus N. \]
Since $N$ is normal in $F$, conjugation by $t$ gives
\[ P_1=gP_2g^{-1}, \qquad P_2=hP_1h^{-1}. \]
Thus $P_1$ and $P_2$ have the same derived length. Since
$N\subseteq P_1\times P_2$ has derived length $3$, both $P_i$ have derived length $3$. By Lemma~\ref{lemma: psi(PGL_2(C))=3}, they are isomorphic to $\Sym_4$. After conjugating independently in the two factors, we may therefore
assume that
\[ P_1=P_2=\Sym_4. \]
Thus $N$ is a subdirect product of $\Sym_4\times\Sym_4$, i.e. $ N \subseteq \Sym_4\times\Sym_4$ and $N$ maps surjectively on the two factors $\Sym_4$.

\vspace{2mm}

1) Let us prove that we have $N = \Sym_4\times\Sym_4$.

Suppose for contradiction that it is not the case.  By Goursat's lemma there exist two normal subgroups $N_1,N_2$ of $\Sym_4$ and an isomorphism $\varphi \colon \Sym_4 / N_1 \to \Sym_4 / N_2$,  such that if $\pi_i \colon \Sym_4 \to \Sym_4 / N_i$, $1 \le i \le 2$, denote the canonical surjections, then $N$ is the fibre product of the morphisms $ \varphi \circ \pi_1$ and $\pi_2 \colon \Sym_4 \to \Sym_4 / N_2$, i.e.
\[ N = \bigl\{ (a,b) \in \Sym_4 \times \Sym_4 \mid  \varphi \circ \pi_1 (a) = \pi_2 (b) \bigr\}. \]
Since the normal subgroups of $\Sym_4$ are
\[ 1,\qquad (\Z/2\Z)^2,\qquad \Alt_4,\qquad \Sym_4, \]
the fibre product is over one of the groups
\[ \Sym_4,\qquad \Sym_3,\qquad \Z/2\Z. \]

Each of these groups has a unique subgroup of index $2$. Consequently, the isomorphism $\varphi$ between the two quotients preserves the sign character, and hence we have
\[ N \subseteq \bigl\{ (a,b) \in \Sym_4 \times \Sym_4 \mid \operatorname{sgn}(a) = \operatorname{sgn}(b) \bigr\}. \]
Let us prove that this inclusion yields $F' \subseteq  \Alt_4 \times \Alt_4$.

We have $ N' \subseteq \Alt_4\times\Alt_4$.

For $n=(a,b)\in N$, we have
\begin{equation}  tnt^{-1} = \bigl( gbg^{-1}, hah^{-1} \bigr). \label{equation: tnt^{-1}} \end{equation}
Conjugation by $g$ and $h$ preserves the sign character, and
$\operatorname{sgn}(a)=\operatorname{sgn}(b)$. Hence
\[ [t,n] \in \Alt_4\times\Alt_4. \]
Since $F$ is generated by $N$ and $t$, it follows that we have actually shown the inclusion 
\[ F ' \subseteq \Alt_4 \times \Alt_4. \]
The latter group has derived length $2$, so $\length(F)\leq3$, a
contradiction. Therefore we have proven the equality $ N=\Sym_4\times\Sym_4$.

\vspace{2mm}

2) Let us now prove that $\tau$ belongs to $F$.

We first claim that we have
\[ N_{\PGL_2(\C)}(\Sym_4)=\Sym_4. \]
Indeed, the centraliser of $\Sym_4$ in $\PGL_2(\C)$ is trivial. For if a non-trivial element of $\PGL_2(\C)$ commuted with $\Sym_4$, its set of fixed points in $\p^1$, which has cardinality at most $2$, would be $\Sym_4$-invariant. This is impossible, since every $\Sym_4$-orbit in $\p^1$ has cardinality at least $6$ by Lemma~\ref{lemma: size orbits for the actions of A4, S4, and A5 on P1}. Now conjugation gives a homomorphism
\[ N_{\PGL_2(\C)}(\Sym_4)\longrightarrow\Aut(\Sym_4) \]
with trivial kernel. Since
\[ \Aut(\Sym_4)=\operatorname{Inn} (\Sym_4), \]
every element of the normalizer differs from an element of $\Sym_4$ by an element centralizing $\Sym_4$. Since $\Sym_4$ is centerless the claim follows.

Since $t$ normalizes $N$, the formula \eqref{equation: tnt^{-1}} above shows that $g,h$ belong to $N_{\PGL_2(\C)}(\Sym_4)$. Hence, they belong to $\Sym_4$. So $(g,h)\in N$ and $\tau=(g,h)^{-1}t$ belongs to $F$.

\vspace{2mm}

Consequently, we have proven the equality
\[ F=(\Sym_4 \times\Sym_4) \rtimes \langle\tau\rangle, \]
as claimed.

\vspace{2mm}

3) Let us now prove that we have $\ell (F) =4$.

The group $F$ is solvable of derived length at most $3+1 = 4$. Hence, it remains to prove that we actually have $\length (F) = 4$. For every $g\in \Sym_4$ the commutator
\[ [(g,1),\tau]=(g,g^{-1}) \]
belongs to the derived subgroup. Hence the derived subgroup projects surjectively onto $\Sym_4$, so it has derived length at least $3$. Therefore the whole group has derived length at least $4$.
\end{proof}

We now conclude the proof in this case. By Lemma~\ref{lemma: length-four subgroups of Aut(P1xP1)}, after conjugation, we have
\[ F=(\Sym_4 \times \Sym_4) \rtimes \langle\tau\rangle. \]
Since $\tau$ exchanges the two rulings, we have
\[ \rk\Pic(\p^1\times\p^1)^F=1, \]
so $\p^1\times\p^1$ is a minimal del Pezzo $F$-surface. By Lemma~\ref{lemma: size orbits for the actions of A4, S4, and A5 on P1}, every $\Sym_4$-orbit in $\p^1$ has cardinality at least $6$. Hence, every $(\Sym_4\times\Sym_4)$-orbit in $\p^1\times\p^1$ has cardinality at least
$6 \times 6=36$. This shows a fortiori that every $F$-orbit in $\p^1\times\p^1$ has cardinality at least~$36$. Since $K_{\p^1\times\p^1}^2=8$, Lemma~\ref{lemma: a minimal del Pezzo G-surface of degree d with no orbit of size <d is G-superrigid} shows that $\p^1\times\p^1$ is $F$-superrigid. Hence we have
\[ \Cent_{\Bir(\p^1\times\p^1)}(F) \subseteq \Aut(\p^1\times\p^1), \]
and in particular $\Cent_{\bp}(F)$ contains no loxodromic element.

\subsection{Hirzebruch surfaces}

We now consider family~(4) of Theorem~\ref{theorem: the eleven families of maximal algebraic subgroups of Bir(P2)}.

\begin{lemma} \label{lemma: Hirzebruch centraliser}
Let $n\geq 1$, and let $F\subseteq\Aut(\F_n)$ be a finite solvable
subgroup of derived length $4$. Then $n$ is odd and every element of $\Cent_{\bp}(F)$ preserves the standard ruling of $\F_n$. In particular, $\Cent_{\bp}(F)$ contains no loxodromic element.
\end{lemma}

\begin{proof}
Let us first prove that $n$ is odd. Suppose for contradiction that $n$ is even. We have
\[ \Aut(\F_n)\simeq \C^{n+1}\rtimes\bigl(\GL_2(\C)/\mu_n\bigr). \]
Since $\C^{n+1}$ has no nontrivial finite subgroup, the projection
\[ F \longrightarrow \GL_2(\C)/\mu_n \]
is injective. Since $n$ is even, we  have
\[ \GL_2(\C)/ \mu_n\simeq \PGL_2(\C) \times \C^* \]
(see \cite[Lemma~3.6]{Urech2021}). A contradiction, because $\psi \bigl( \PGL_2(\C)\times\C^* \bigr) =3$. Thus $n$ is odd.

Let now $\varphi\in\Cent_{\bp}(F)$. Since $\varphi$ commutes with $F$, it is an $F$-equivariant birational self-map of $\F_n$. By
Proposition~\ref{proposition: equivariant self-maps of odd Hirzebruch surfaces}, the map $\varphi$ preserves the standard ruling of $\F_n$, and hence it is not loxodromic (see Lemma~\ref{lemma: rational fibration}).
\end{proof}

\subsection{Exceptional conic bundles}

We now consider family~(5) of Theorem~\ref{theorem: the eleven families of maximal algebraic subgroups of Bir(P2)}, consisting of exceptional conic bundles with at least four singular fibres.

Let $(S,\pi)$ be such an exceptional conic bundle, and let
$\Delta\subseteq\p^1$ be the set of points over which $\pi$ has a singular fibre. Write $|\Delta|=2n$, where $n\geq2$.

We use the notation $\rho$, $\sigma$ and $T$ introduced in Case~(5) of the proof of Theorem~\ref{theorem: psi_{finite}(Bir(P2)) =4}. Thus $\rho \colon \Aut(S,\pi) \to H_\Delta$ is the action on the base, $\sigma \colon \Aut(S,\pi) \to \Z/2\Z$ is the induced action on the pair of exceptional sections $\{s_1,s_2\}$, and $T=\ker(\rho,\sigma)\simeq\C^*$ is a one-dimensional torus.

Recall that we have the following short exact sequence:
\begin{equation} 1 \longrightarrow T \longrightarrow \Aut(S,\pi)  \xrightarrow{(\rho,\sigma)} H_\Delta \times \Z/2\Z \longrightarrow 1. \label{1st short exact sequence} \end{equation}
We will also use the following exact sequence from \cite[Theorem~2]{Blanc2009}:
\begin{equation}  1 \longrightarrow T \rtimes \Z / 2 \Z \longrightarrow \Aut (S,\pi) \longrightarrow H_\Delta \longrightarrow 1. \label{2nd short exact sequence} \end{equation}

\vspace{2mm}

Let $F \subseteq \Aut(S,\pi)$ be a finite solvable group of derived length $4$. We will show that $\Cent_{\bp}(F)$ contains no loxodromic element.

\vspace{2mm}

Taking the intersection with $F$, the short exact sequence \eqref{1st short exact sequence}  shows that $P:= \rho (F)  \subseteq  H_{\Delta} \subseteq  \PGL_2 (\C) $ has derived length $3$, and hence Lemma~\ref{lemma: psi(PGL_2(C))=3} gives
\[ P \simeq \Sym_4. \]
Since $\Delta$ is $P$-invariant, Lemma~\ref{lemma: size orbits for the actions of A4, S4, and A5 on P1} gives $|\Delta|\geq 6$, i.e.
\[ n \ge 3.\]

We first prove the result under the additional assumption that the conic bundle $(S,\pi)$ is $F$-minimal. We will then reduce the general case to this one.

\begin{lemma}[Fibration preservation for $F$-minimal exceptional conic bundles]
\label{lemma: minimal exceptional conic bundle centraliser}
Let $\pi \colon S\to\p^1$ be an exceptional conic bundle with at least $2n \ge 4$ singular fibres, and let $F \subseteq \Aut(S,\pi)$ be a finite solvable subgroup of derived length $4$. Assume that $(S,\pi)$ is $F$-minimal. Then every $F$-equivariant birational self-map
\[ \varphi \colon S \dashrightarrow S \]
preserves the fibration $\pi$. In particular, it is not loxodromic.
\end{lemma}

\begin{proof}
The short exact sequence \eqref{1st short exact sequence} shows that $D^3F$ is contained in $T$. Choose
\[ 1\neq t\in D^3F\subseteq T. \]
The element $t$ acts trivially on the base, and its two fixed points
on every smooth fibre are the points lying on $s_1$ and $s_2$.

Choose a reduced $F$-equivariant Sarkisov factorization of $\varphi$. We prove along this factorization that every intermediate conic bundle is again an $F$-minimal exceptional conic bundle with the same number $2n$ of singular fibres, and with two $t$-fixed sections of self-intersection $-n$ whose intersections with each smooth fibre are precisely the two fixed points of $t$.

Suppose that this holds for the current conic bundle, which we again denote by $(S,\pi)$, with distinguished sections $s_1,s_2$. Recall that we have proven above that we actually have $n \geq 3 $. Hence, we have
\[ K_S^2=8-2n \leq 2. \]
A link of type~III starting from a conic bundle can occur only from
$\F_1$ or when $K_S^2\in\{3,5,6\}$; see \cite[Theorem~2.6]{Iskovskikh1996}. Hence no link of type~III can occur. Moreover, by adjunction,
\[ (-K_S)\cdot s_i=s_i^2+2=2-n<0, \]
so $S$ is not a del Pezzo surface. By
\cite[Lemma~7.1.1]{Blanc2009}, an $F$-minimal conic bundle admitting a second $F$-invariant conic-bundle structure is a del Pezzo surface. Thus no link of type~IV can occur. Consequently, any next link is of type~II over the same base.

Such a link is an elementary transformation centred at an $F$-orbit $\Omega$ of points lying on pairwise distinct smooth fibres.

Since $t$ preserves $\Omega$ and acts trivially on the base, every point of $\Omega$ is fixed by $t$. Hence
\[ \Omega\subseteq s_1\cup s_2. \]

Suppose first that $\ker(\rho|_F)$ is not contained in $T$. There exists then an element
\[ \tau \in \ker(\rho|_F) \setminus T,\]
i.e.\ an element of $F$ which acts trivially on the base and exchanges the two sections $s_1$ and $s_2$.  For any $x\in\Omega$, the points $x$ and $\tau(x)$ are distinct points of $\Omega$ lying on the same fibre, a contradiction. Thus no link can occur in this case.

Suppose now that  $\ker(\rho|_F)$ is contained in $T$. Then $\sigma$ factors through a homomorphism
\[ \bar\sigma\colon P\simeq\Sym_4\longrightarrow\Z/2\Z. \]
Since $(S,\pi)$ is $F$-minimal, Lemma~\ref{lemma: characterization of minimal G-conic bundles via twistedness}
shows that all its singular fibres are $F$-twisted. If $\bar\sigma$
were trivial, every element of $F$ would preserve each exceptional section and hence each component of every singular fibre, a contradiction. Therefore $\bar\sigma$ is nontrivial, and hence it is the sign character.

Let $x\in\Omega$ and put $p:=\pi(x)$. If $\Stab_P(p)$ contained an odd permutation, a lift to $F$ would fix $p$ on the base and exchange $s_1$ and $s_2$. It would therefore send $x$ to a second point of $\Omega$ on the same fibre, a contradiction. Hence we have
\[ \Stab_P(p)\subseteq\Alt_4. \]
The projection $\pi \colon S \to \p^1$ induces a bijection $\Omega\to P \cdot p$. Since left multiplication by any fixed odd element of $P$ exchanges the even and odd cosets of $\Stab_P(p)$, exactly half of the points of $\Omega$ lie on $s_1$ and half lie on $s_2$.

An elementary transformation centred at a point of $s_1$ decreases $s_1^2$ by one and increases $s_2^2$ by one, and conversely for a centre point on $s_2$. Since the centre is balanced, the two self-intersections remain equal to $-n$. The number of singular fibres is unchanged, so the target is again an exceptional conic bundle with $2n$ singular fibres.

The transformed sections are $t$-fixed and, on every smooth fibre, their intersection points are again precisely the two fixed points of $t$.  Indeed, near a point of the centre we may choose local coordinates $(u,v)$ such that
\[ \pi(u,v)=u,\qquad t(u,v)=(u,\lambda v),\qquad \lambda\neq1. \]
On the exceptional divisor of the blow-up, $t$ acts by
$[u:v] \mapsto [u:\lambda v]$. Its two fixed points correspond to the directions of the fixed section and of the fibre. After contracting the strict transform of the fibre, these give precisely the intersection points of the two transformed sections with the new fibre.

The actions of $F$ on the base and on the two transformed sections are unchanged. Since the target of a Sarkisov link is $F$-minimal, the induction continues. It follows that every link in the factorization is of type~II over the same base. Hence $\varphi$ preserves the fibration $\pi$.
\end{proof}

We are now ready to state and prove the main result of this section.

\begin{lemma} \label{lemma: exceptional conic bundle centraliser}
Let $(S,\pi)$ be an exceptional conic bundle with $2n \ge 4$ fibres, and let $F\subseteq\Aut(S,\pi)$ be a finite solvable subgroup of derived length $4$. Then $\Cent_{\bp}(F)$ contains no loxodromic element.
\end{lemma}

\begin{proof}
In view of Lemma~\ref{lemma: minimal exceptional conic bundle centraliser}, it is enough to reduce to the case where the exceptional conic bundle $(S,\pi)$ is an $F$-Mori fibre space. We make this reduction in three steps. We will successively reduce to the following cases:

\vspace{2mm}

1) The kernel $\ker(\rho|_F)$ is contained in $T$ (this is reduction 1 below). Hence, taking the intersection of the second short exact sequence  \eqref{2nd short exact sequence} with $F$ will yield the short exact sequence
\begin{equation} 1 \longrightarrow T \cap F  \longrightarrow F \longrightarrow  P \longrightarrow  1. \label{3rd short exact sequence}  \end{equation}

2) The map $\sigma|_F \colon F \to \Z / 2 \Z$ is surjective (this is reduction 2 below). Hence, we have
\begin{equation} \sigma|_F = \sgn \circ \rho|_F   \label{equation: sigma|_F = sgn o rho|_F} \end{equation}
where $\sgn \colon P \simeq \Sym_4  \to \Z / 2 \Z$ is the sign character.

\vspace{2mm}

3) The exceptional conic bundle $(S,\pi)$ is an $F$-Mori fibre space (this is reduction 3 below).

\vspace{2mm}

\noindent \uline{Reduction 1}. Let us reduce to the case where $\ker(\rho|_F) \subseteq T$.

If $\ker(\rho|_F)$ were not contained in $T$, there would exist an element
\[ \tau \in \ker(\rho|_F) \setminus T,\]
i.e.\ an element of $F$ which acts trivially on the base and exchanges the two sections $s_1$ and $s_2$. By
\cite[Lemma~4.3.3(2), p.~265]{Blanc2009}, the element $\tau$ is an involution fixing an irreducible curve
\[ C\longrightarrow\p^1\]
which is a double cover branched precisely over $\Delta$. Riemann--Hurwitz therefore gives
\[ g(C)=n-1\geq2. \]
Moreover, $\tau$ restricts to a nontrivial involution on a general fibre, whose two fixed points are precisely the points lying on $C$. Thus $C$ is the unique fixed curve dominating $\p^1$, while every other fixed curve is contained in a fibre and hence rational. Therefore Corollary~\ref{corollary: central involution with a high-genus fixed curve} shows that every element of $\Cent_{\bp}(F)$ is non-loxodromic.

We may henceforth assume that we have $\ker(\rho|_F) \subseteq T$, and hence we have the short exact sequence \eqref{3rd short exact sequence}.

\vspace{2mm}

\noindent \uline{Reduction 2}. Let us reduce to the case where $\sigma|_F \colon F \to \Z / 2 \Z$ is surjective.

Otherwise, every element of $F$ would preserve each of the two exceptional sections $s_1$ and $s_2$. The birational morphism
\[ \eta \colon S \longrightarrow \F_n \]
obtained by contracting in every singular fibre the component which does not meet $s_1$ is $F$-equivariant. Hence, for every $\varphi\in\Cent_{\bp}(F)$, the map
\[ \eta \varphi\eta^{-1} \colon \F_n\dashrightarrow\F_n \]
is $F$-equivariant. Lemma~\ref{lemma: Hirzebruch centraliser} shows that it preserves the standard ruling. Thus $\varphi$ preserves a rational fibration and is not loxodromic by Lemma~\ref{lemma: rational fibration}. We may therefore assume that $\sigma|_F \colon F \to \Z / 2 \Z$ is surjective, i.e. the equality \eqref{equation: sigma|_F = sgn o rho|_F} is satisfied.

\vspace{2mm}

\noindent \uline{Reduction 3}. Let us finally reduce to the case where the exceptional conic bundle $(S,\pi)$ is an $F$-Mori fibre space.

Assume that the $F$-conic bundle $\pi \colon S \to \p^1$ is not minimal (otherwise, there is nothing to do). Then, for each $F$-orbit of an $F$-untwisted fibre $E_1 + E_2$ (see Definition~\ref{definition: G-twisted and untwisted singular fibre}), choose one of the two orbits $g.E_1$, $g \in F$, or $g.E_2$, $g \in F$ (both of these orbits consist of disjoint $(-1)$-curves) and blow down each of its $(-1)$-curves. In this way, we obtain a new $F$-conic bundle $(S_0, \pi_0)$ with an $F$-equivariant morphism of conic bundles
\[ \eta \colon (S, \pi) \to (S_0, \pi_0).\]
The $F$-conic bundle $(S_0, \pi_0)$ is minimal by construction. Indeed, we have actually performed the usual relative $F$-equivariant minimal model program for conic bundles.

Let us now check that the conic bundle $(S_0, \pi_0)$ is still exceptional. By induction, it is enough to show that when we contract a single $F$-orbit
\[ \Omega = \{ g.E_i, \, g \in F \}, \] 
where $i=1$ or $2$, and as above $E_1 + E_2$ is an $F$-untwisted fibre, then the new conic bundle $(S' , \pi')$ is again exceptional. But the cardinality of the orbit $\Omega$ is even (because if we set $p:= \pi (E_i) \in \p^1$, then we have $| \Omega | = | F. p |$ which is even by  Lemma~\ref{lemma: size orbits for the actions of A4, S4, and A5 on P1}). Let $2k$ be the cardinality of this orbit.

We claim that half of the curves in the orbit $\Omega$ meet the section $s_1$, and that the other half meet $s_2$. Let us prove it. Since $\ker(\rho|_F)\subseteq T$ preserves each component of every singular fibre, the action of $F$ on these components factors through $P$. Set
\[ \Stab_P(E_i):=\{g\in P\mid g(E_i)=E_i\}. \]
We have a bijection
\[ P / \Stab_P(E_i )  \to \Omega, \quad g. \Stab_P(E_i ) \mapsto g.E_i.\]
Each element of $\Stab_P(E_i)$ is even, i.e.\ belongs to the kernel of the sign character $\sgn \colon P \simeq \Sym_4  \to \Z / 2 \Z$. Now, left multiplication by any fixed odd element of $P$ exchanges the even and odd cosets of $\Stab (E_i )$. This proves the claim.

It follows that $(S',\pi')$ has $2(n-k)$ singular fibres, and that the new sections $s'_1$ and $s'_2$, which are the images of $s_1$ and $s_2$ under the contraction of the curves in $\Omega$, satisfy $(s'_i)^2=-(n-k)$ (because $(s_i)^2=-n$ and $s_i$ meets exactly $k$ curves which are contracted), for $i=1,2$.

We have therefore shown that if $(S_0,\pi_0)$ has $2n_0$ singular fibres, then it admits two sections $s_{0,1}$ and $s_{0,2}$ satisfying
\[ s_{0,1}^2=s_{0,2}^2=-n_0. \]
It remains to show that $n_0\neq0$. Suppose for contradiction that $n_0=0$. Then $\pi_0$ is a $\p^1$-bundle with two disjoint $0$-sections, and hence $S_0\simeq\F_0$. This gives
\[ F\subseteq\Aut(\F_0,\pi_0)\simeq\PGL_2(\C)\times\PGL_2(\C), \]
contradicting $\psi\bigl(\PGL_2(\C)\times\PGL_2(\C)\bigr)=3$ and $\ell(F)=4$.

Thus $n_0\neq0$, and hence $(S_0,\pi_0)$ is exceptional. As above, we have $n_0\ge3$, since the set of singular fibres projects onto a finite union of $\Sym_4$-orbits on $\p^1$, so Lemma~\ref{lemma: size orbits for the actions of A4, S4, and A5 on P1} gives $2n_0\ge6$.

\vspace{2mm}

We are therefore reduced to the case where the exceptional conic bundle is $F$-minimal, and the result follows from Lemma~\ref{lemma: minimal exceptional conic bundle centraliser}.
\end{proof}

\subsection{The Fermat cubic}

By Case~(8) in the proof of Theorem~\ref{theorem: psi_{finite}(Bir(P2)) =4}, a finite solvable subgroup of derived length $4$ can occur in this family only when $S$ is the Fermat cubic
\[ X^3+Y^3+Z^3+W^3=0. \]

In this case
\[ \Aut(S)=V\rtimes\Sym_4, \qquad V\simeq(\Z/3\Z)^3. \]

 We may identify $V$ with
 \[ V = \{ (a_1,a_2,a_3,a_4) \in  (\FFin_3)^4, \; a_1 + a_2 + a_3 +a_4 = 0\},\]
 where each element $(a_1,a_2,a_3,a_4) \in V$ acts on $S$ as the automorphism
 \[ (X,Y,Z,W) \mapsto ( \omega^{a_1} X,  \omega^{a_2} Y,  \omega^{a_3} Z,  \omega^{a_4} W)\]
 of $S$ ($\omega$ being a primitive third root of unity). The group $\Sym_4$ acts on $S$ (and on $V$) by permutation of the coordinates.
 
We begin with the following elementary observation.
 
\begin{lemma}
\label{lemma: The S4-module V is irreducible}
The $\Sym_4$-module $V$ is irreducible.
\end{lemma}

\begin{proof}
A nonzero vector $a=(a_1,a_2,a_3,a_4)\in V$ cannot have all four
coordinates equal. Choose coordinates $i,j$ with $a_i\neq a_j$. Then
\[ a-(i,j)a=(a_i-a_j)(e_i-e_j) \]
is nonzero. The $\Sym_4$-orbit of $e_i-e_j$ spans $V$.
\end{proof}

We now prove the following result, which is analogous to Lemma~\ref{lemma: length-four subgroup of H_{216}}, both in its statement and in its proof.

\begin{lemma}
\label{lemma: length-four subgroup of Fermat cubic}
We have $\length(\Aut(S))=4$. Moreover, if
$F\subseteq\Aut(S)$ is a subgroup of derived length $4$, then
$F=\Aut(S)$.
\end{lemma}

\begin{proof}
1) We begin by proving that the derived length of $\Aut(S)=V\rtimes\Sym_4$ is equal to~$4$. The derived series of $\Sym_4$ is
\[ \Sym_4\triangleright\Alt_4\triangleright (\Z/2\Z)^2\triangleright 1. \]
We claim that the derived series of $V\rtimes\Sym_4$ is
\[ V\rtimes\Sym_4 \; \triangleright \; V\rtimes\Alt_4 \;  \triangleright \; V \rtimes (\Z/2\Z)^2 \; \triangleright \;  V \; \triangleright \; 1. \]
Indeed, since $V$ is abelian, for every subgroup $H \subseteq \Sym_4$ we have
\[ (V\rtimes H)'=[V,H]\rtimes H'. \]

Recall that $[V, H]$ denotes the subgroup of $V$ generated by the elements of the form $h(v) -v$, $h \in H$, $v \in V$ (see the proof of Lemma~\ref{lemma: length-four subgroup of H_{216}}).

It is therefore enough to show that $[V,(\Z/2\Z)^2]=V$. Since
$(\Z/2\Z)^2\triangleleft\Sym_4$, the subgroup $[V,(\Z/2\Z)^2]$ is a
$\Sym_4$-submodule of $V$. By Lemma~\ref{lemma: The S4-module V is irreducible}, it is enough to show that $[V,(\Z /2\Z)^2]\neq 0$. Taking
\[ h=(1,2)(3,4)\in (\Z/2\Z)^2,\qquad v=e_1-e_2 \in V, \]
we have $h(v)=-v$, and hence
\[ h(v)-v=-2v=v\neq0. \]
Thus $\length(\Aut(S))=4$.

2) We finally prove that if $F \subseteq \Aut(S)$ satisfies $\ell(F)=4$, then we have $F  = \Aut(S)$.

Let $P$ be the image of $F$ under the projection
$\Aut(S)=V\rtimes\Sym_4\longrightarrow\Sym_4$. Since $F\cap V$ is abelian, the group $P$ has derived length $3$, and hence
\[ P=\Sym_4. \]
The intersection $F\cap V$ is therefore a $\Sym_4$-submodule of $V$. By Lemma~\ref{lemma: The S4-module V is irreducible}, it is either $0$ or $V$. If $F\cap V=0$, then $F\simeq\Sym_4$, and hence $\length(F)=3$, a contradiction. Thus $F\cap V=V$. Since the image of $F$ in $\Sym_4$ is $\Sym_4$, we obtain $ F=V\rtimes\Sym_4=\Aut(S)$.
\end{proof}

We can now prove the main result of this case

\begin{lemma}
Let $S\subseteq\p^3$ be the Fermat cubic and let
$F\subseteq\Aut(S)$ be a finite solvable group of derived length $4$. Then $\Cent_{\bp}(F)$ contains no loxodromic element.
\end{lemma}

\begin{proof}
By Lemma~\ref{lemma: length-four subgroup of Fermat cubic},
we have $F=\Aut(S)$. By \cite[Theorem 1]{Blanc2009}, $S$ is a minimal $F$-del Pezzo surface which is superrigid. Hence, we have $\Cent_{\bp}(F) \subseteq \Aut (S)$, and in particular $\Cent_{\bp}(F)$ contains no loxodromic element.
\end{proof}

\begin{remark}
Let $F = \Aut (S)$. The $F$-superrigidity of $S$ can also be deduced directly from Lemma~\ref{lemma: a minimal del Pezzo G-surface of degree d with no orbit of size <d is G-superrigid} by verifying its orbit condition. Let $p=[X:Y:Z:W] $ be a point of the cubic. At least two of its coordinates are nonzero. Up to permutation of the coordinates, we may assume that $X,Y\neq 0$. The three points
\[ [\omega^aX:Y:Z:W],\qquad a \in \FFin_3, \]
are distinct. Hence every $F$-orbit in $S$ has cardinality at least $3$. Since $K_S^2=3$, the result follows.
\end{remark}

\subsection{Degree-one del Pezzo surfaces}

Let $S$ be a del Pezzo surface of degree $1$, and let $\beta\in\Aut(S)$ be the Bertini involution. Recall from Case~(10) in the proof of Theorem~\ref{theorem: psi_{finite}(Bir(P2)) =4} that $\beta$ is central in $\Aut(S)$, that every automorphism of $S$ fixes the isolated fixed point $p$ of $\beta$, and that the tangent representation
\[ \rho_p\colon\Aut(S)\longrightarrow\GL(T_pS)\simeq\GL_2(\C) \]
is faithful.

\begin{lemma} \label{lemma: A solvable subgroup of derived length 4 of the automorphism group of a del Pezzo surface of degree 1 contains the Bertini involution}
Let $S$ be a del Pezzo surface of degree $1$ and let $F$ be a solvable subgroup of derived length $4$ of $\Aut (S)$. Then $F$ contains the Bertini involution $\beta$.
\end{lemma}

\begin{proof}
Let $\rho_p \colon \Aut (S) \to \GL (T_pS)$ be the tangent representation.

1) Let us prove that we have $\rho_p ( \beta ) = -I$.
 
Otherwise, $\rho_p ( \beta )$ being an involution of $\GL (T_pS) \simeq \GL_2 (\C)$ it would be conjugate to the matrix $\smat{1}{\hspace{2mm}0}{0}{-1}$. But $\beta$ is a central element of $\Aut (S)$. In particular, it commutes with each element of $F$, and hence $\rho_p (F)$ is contained in
\[  \Cent_{\GL_2 (\C) }(\smat{1}{\hspace{2mm}0}{0}{-1} ) = \{ \smat{a}{0}{0}{b}, \; a,b \in \C^* \}.\]
A contradiction, because then $\rho_p (F)$ would be commutative and we have assumed that its derived length is $4$.

2) By Lemma~\ref{lemma: A solvable subgroup of derived length 4 of GL_2(C) contains the subgroup {I,-I}}, $\rho_p (F)$ contains the matrix $-I$. Since $\rho_p$ is injective, $F$ contains $\beta$ by the first point.
\end{proof}

\begin{lemma}
Let $S$ be a del Pezzo surface of degree $1$ and let $F \subseteq \Aut (S) $ be a finite solvable group of derived length $4$. Then $ \Cent_{\bp}(F)$ contains no loxodromic element.
\end{lemma}

\begin{proof}
By Lemma~\ref{lemma: A solvable subgroup of derived length 4 of the automorphism group of a del Pezzo surface of degree 1 contains the Bertini involution}, the group $F$ contains the Bertini involution $\beta$. Hence every element of $ \Cent_{\bp}(F)$ commutes with $\beta$. It is well-known that the unique positive-genus component of $ \Fix (\beta) $ is a smooth curve of genus $4$ (see \cite[p.~12]{BayleBeauville2000}). Therefore Corollary~\ref{corollary: central involution with a high-genus fixed curve} applies and excludes loxodromic elements.
\end{proof}

\subsection{\((\Z / 2 \Z)^2\)-conic bundles}
Finally, we consider family~(11) of Theorem~\ref{theorem: the eleven families of maximal algebraic subgroups of Bir(P2)}. We briefly recall the features of $(\Z/2\Z)^2$-conic bundles that will be used below. Let $\pi\colon S\to\p^1$ be such a conic bundle. The subgroup of automorphisms acting trivially on the base is a Klein four group.

Set
\[ V :=\Aut(S/ \p^1)\simeq(\Z/2\Z)^2. \]
For each nontrivial involution $\sigma\in V$, the fixed locus of
$\sigma$ has a unique irreducible component $C_\sigma$ dominating the base, and
\[ \pi|_{C_\sigma}\colon C_\sigma \to \p^1 \]
is a double cover branched over a nonempty set
$A_\sigma \subset \p^1$ of even cardinality; see Definition~\ref{definition: (Z/2Z)^2-conic bundle}. Moreover, conjugation by an element of $\Aut(S,\pi)$ permutes the three nontrivial elements of $V$, together with their corresponding fixed curves and branch sets.

\begin{lemma}[Central fibrewise involution]
\label{lemma:central fibrewise involution}
Let $\pi\colon S\to\p^1$ be a $(\Z/2\Z)^2$-conic bundle, and let
\[ 1 \longrightarrow V \longrightarrow\Aut(S,\pi) \xrightarrow{\rho}H_V\longrightarrow1, \qquad V\simeq(\Z/2\Z)^2. \]
be the exact sequence induced by the action on the base. Let $ F \subseteq \Aut(S,\pi)$ be a finite solvable subgroup of derived length $4$. Then
\[ \rho(F) \simeq \Sym_4 \qquad \text{and} \qquad V \cap \mathrm Z(F)\neq1. \]
In particular, $F$ contains a nontrivial fibrewise involution which is central in $F$.
\end{lemma}

\begin{proof}
Put
\[ W:=F\cap V=\ker(\rho|_F), \qquad P:=\rho(F). \]
We have a short exact sequence
\[ 1 \longrightarrow W \longrightarrow F \longrightarrow P
\longrightarrow 1. \]
Since $W$ is abelian and $P \subseteq \PGL_2(\C)$, we obtain
\[ 4=\ell(F) \leq \ell (W)+\ell (P) \leq1+3=4. \]
Thus $W\neq1$ and $\ell(P)=3$, and hence we have $P\simeq\Sym_4$ by 
Lemma~\ref{lemma: psi(PGL_2(C))=3}.

If $|W|=2$, let $\sigma$ be its unique nontrivial element. Since
$W$ is normal in $F$, conjugation by any element of $F$ preserves $W$, and hence fixes $\sigma$. Thus $\sigma\in V\cap\mathrm Z(F)$.

We may therefore assume that $W=V$. In particular $V \subseteq F$ and
\[ F / V \simeq P \simeq \Sym_4. \]
Since $V$ is abelian, conjugation by elements of $F$ on $V$ factors through $F/ V$ and we obtain a homomorphism $\alpha \colon P \simeq F/ V \longrightarrow \Aut( V) \simeq \Sym_3$ which records how conjugation by elements of $F$ permutes the three nontrivial elements of $V$.

Suppose first that $\alpha$ is not surjective. Its image is then
either trivial or of order $2$: it cannot have order $3$, since
$\Sym_4^{\mathrm{ab}}\simeq\Z/2\Z$. In either case, the image of
$\alpha$ fixes one of the three nontrivial elements of $V$. Hence there exists
\[ 1 \neq \sigma \in V \]
which is fixed by conjugation by every element of $F$, and therefore $\sigma \in V \cap \mathrm Z(F)$.

It remains to consider the case where $\alpha$ is surjective. Put
\[ K:=\ker(\alpha). \]
Then $K$ is the normal Klein four subgroup of $P \simeq \Sym_4$, so
\[ K= D^2P \simeq(\Z/2\Z)^2. \]
Let
\[ B := (\rho|_F)^{-1}(K). \]
Since $K$ acts trivially on $V$ by conjugation (because $K =\ker(\alpha)$), we have
\[ V \subseteq \mathrm Z(B), \qquad B/ V \simeq K\simeq(\Z/2\Z)^2. \]
Choose $x,y\in B$ whose images generate $B/V$. Then $B$ is generated by $V, x,y$, and $V$ is central in $B$. Consequently,
\[ B' = \langle[x,y] \rangle \subseteq V, \]
so $|B'| \leq 2$.

Since $\rho|_F \colon F \to P$ is surjective,
\[ \rho(D^2F)=D^2P=K, \]
and hence $D^2F \subseteq B$. It follows that
\[ 1 \neq D^3F = (D^2F)' \subseteq B', \]
where the first inequality follows from $\ell(F)=4$. Thus $D^3F$
has order $2$. Since $D^3F$ is characteristic in $F$, its unique
nontrivial element is central in $F$. Moreover,
\[ D^3F \subseteq B' \subseteq V. \]
Hence $V \cap \mathrm Z(F) \neq1$.
\end{proof}

\begin{lemma}
\label{lemma:  the centraliser of a (Z/2Z)^2-conic bundle contains no loxodromic element}
Let $\pi\colon S\to\p^1$ be a $(\Z/2\Z)^2$-conic bundle such that
$S$ is not a del Pezzo surface, and let $F\subseteq\Aut(S,\pi)$ be a finite solvable subgroup of derived length $4$. Then $\Cent_{\bp}(F)$ contains no loxodromic element.
\end{lemma}

\begin{proof}
Let $P:=\rho(F)\subseteq\PGL_2(\C)$.  By Lemma~\ref{lemma:central fibrewise involution}, we have $P\simeq \Sym_4$ and there exists $ 1\neq \sigma \in V \cap \mathrm Z(F)$. Recall that the involution $\sigma$ fixes pointwise an irreducible curve $C_\sigma$ such that
\[ \pi|_{C_\sigma}\colon C_\sigma\longrightarrow\p^1 \]
is a double cover branched over a nonempty set $A_\sigma \subseteq \p^1$ of even cardinality; see Definition~\ref{definition: (Z/2Z)^2-conic bundle}. Since a nontrivial involution of a smooth fibre has exactly two fixed points, $C_\sigma$ is the unique irreducible component of $\Fix(\sigma)$ dominating the base.

Since $\sigma$ is central in $F$, every element of $F$ preserves
$C_\sigma$. Indeed, if $g\in F$, then
\[ g(C_\sigma)\subseteq\Fix(\sigma), \]
and $g(C_\sigma)$ dominates the base. Since $C_\sigma$ is the unique irreducible component of $\Fix(\sigma)$ dominating the base, we have
\[ g(C_\sigma)=C_\sigma. \]
It follows that $P=\rho(F)$ preserves the branch set $A_\sigma$.

Since $P\simeq\Sym_4$ and $A_\sigma$ is nonempty,
Lemma~\ref{lemma: size orbits for the actions of A4, S4, and A5 on P1} gives $ |A_\sigma|\geq 6$. By the Riemann--Hurwitz formula,
\[ g(C_\sigma) =\frac{|A_\sigma|}{2}-1 \geq2. \]
Moreover, every other irreducible curve contained in $\Fix(\sigma)$ is contained in a fibre of $\pi$, and is therefore rational. Thus
$C_\sigma$ is the unique positive-genus component of $\Fix(\sigma)$. Finally, every element of $\Cent_{\bp}(F)$ commutes with $\sigma$. Corollary~\ref{corollary: central involution with a high-genus fixed curve} therefore shows that no such element is loxodromic.
\end{proof}

\section{Derived-length bounds by dynamical type}
\label{sec:dynamical-bounds}

We now establish the remaining estimates needed for the proof of
Theorem~\ref{theorem: main}. We treat successively solvable subgroups containing a loxodromic element, subgroups preserving a rational fibration, and automorphism groups of Halphen surfaces.

\subsection{Solvable groups with loxodromic elements} \label{subsection: Solvable groups with loxodromic elements}

We will use the following theorem of Cantat (see \cite[Theorem A.1]{DelzantPy2012}).

\begin{theorem}[Cantat]
\label{theorem: Cantat-Delzant-Py}
Let $G\subseteq\bp$ be a group containing a loxodromic element and an infinite bounded normal subgroup. Then $G$ is conjugate to a subgroup of
\[ \Aut((\C^*)^2)=(\C^*)^2\rtimes\GL_2(\Z). \]
\end{theorem}

\begin{lemma}
\label{lemma: second-derived fixes the axis}
Let $f \in \bp$ be loxodromic, and let $H \subseteq \Stab(\Ax(f))$ be a subgroup. Then
\[ D^2H \subseteq \Fix(\Ax(f)). \]
\end{lemma}

\begin{proof}
The action of $H$ on the axis defines a homomorphism
\[ \rho \colon H \longrightarrow {\rm Isom}(\Ax(f)) \simeq {\rm Isom}(\R). \]
Its kernel is precisely $H \cap \Fix(\Ax(f))$. We have
${\rm Isom}(\R) \simeq \R \rtimes \Z/2\Z$, and in particular 
$D^2{\rm Isom}(\R)=1$. It follows that
\[ \rho(D^2H) \subseteq D^2\rho(H) =1. \]
Hence $D^2H \subseteq \ker(\rho)=H \cap \Fix(\Ax(f))$, as required.
\end{proof}

\begin{theorem}
\label{theorem: loxodromic case}
Let $H \subseteq \bp$ be a solvable subgroup containing a loxodromic element. Then we have
\[ \length(H) \leq 5. \]
\end{theorem}

\begin{proof}
Let $f \in H$ be loxodromic.  We first prove that $H \subseteq \Stab(\Ax(f)).$

Let $h\in H$ and put $g=hfh^{-1}$. Suppose that
$\Ax(g)\neq\Ax(f)$. By \cite[Lemma~6.3]{Urech2021}, the loxodromic elements $f$ and $g$ have no common fixed point on the boundary of the Picard--Manin hyperbolic space. Suitable powers of $f$ and $g$ therefore generate a non-abelian free group by the usual ping-pong argument. This is impossible since $H$ is solvable. Hence $\Ax(g)=\Ax(f)$, and therefore $h$ preserves $\Ax(f)$ as claimed.

By Lemma~\ref{lemma: second-derived fixes the axis}, we have
\[ D^2H\subseteq\Fix(\Ax(f)). \]

Assume first that $\Fix(\Ax(f))$ is infinite. Recall that a subgroup of $\bp$ admitting a fixed point in $\H^\infty$ is bounded (see e.g. \cite[Lemma~2.6]{Urech2021}). Hence  $\Fix(\Ax(f))$ is bounded. Since it is also normal in $\Stab(\Ax(f))$, Theorem~\ref{theorem: Cantat-Delzant-Py} shows that $\Stab(\Ax(f))$ is conjugate to a subgroup of the monomial group
\[ \Aut \bigl( (\C^*)^2 \bigr) =(\C^*)^2\rtimes\GL_2(\Z). \]
Since $H$ is contained in $\Stab(\Ax(f))$, we get
\[ \length(H) \leq \psi \bigl( \Aut \bigl( (\C^*)^2 \bigr) \bigr) \leq \psi\bigl( (\C^*)^2 \bigr) +\psi \bigl( \GL_2(\Z) \bigr) \leq 1+4=5. \]

We may therefore assume that $\Fix(\Ax(f))$ is finite. Then $D^2H$ is finite.
By Theorem~\ref{theorem: psi_{finite}(Bir(P2)) =4}, we have
\[ \length(D^2H) \leq 4, \]
and therefore
\[ \length(H) \leq 2+\length(D^2H) \leq 6. \]
Suppose, for a contradiction, that $\length(H)=6$. Then
\[ F:=D^2H \]
is a finite solvable subgroup of $\bp$ of derived length $4$.
Since $D^2H$ is characteristic in $H$, the loxodromic element $f$ normalizes $F$. Conjugation by $f$ therefore induces an automorphism of the finite group $F$. Since $\Aut(F)$ is finite, some positive power of this automorphism is trivial.
Equivalently, there exists an integer $m \geq 1$ such that
\[ f^m \in \Cent_{\bp}(F). \]
But $f^m$ is still loxodromic, contradicting Theorem~\ref{theorem: the centraliser of a solvable group of derived length 4 has no loxodromic element}. Thus $\length(H) \neq 6$, and consequently $\length(H) \leq 5$.
\end{proof}

\subsection{The Jonqui\`eres group}

We next establish the sharp bound for solvable subgroups preserving a rational fibration. Let $p_0=[1:0:0]$, and denote by
$\Jonq_{p_0}\subseteq\bp$ the subgroup preserving the pencil of lines through $p_0$.

The symmetric group $\Sym_4$ can be realized in $\PGL_2(\C)$ in the following way:

\begin{definition} \label{definition: a realization of S4 in PGL(2,C)}
Set $\Sym_4 := \big\langle y \mapsto iy, y \mapsto \frac{y+1}{y-1} \big\rangle \subseteq \PGL_2 (\C)$.
\end{definition}

Recall that this realization is unique up to conjugation
(see Lemma~\ref{lemma: finite subgroups of PGL_2(C)}). As usual, the group $\Sym_4$ might also be seen as a subgroup of $\bp$ via the natural embedding of $\PGL_2 (\C)$ in $\bp$ given by $(y \mapsto g(y) ) \mapsto (x, g(y) )$.

\begin{lemma} \label{lemma: a subgroup of J whose derived length is 5}
The group
\[ \Aff_1 (\C (y) ) \rtimes \Sym_4 = \{ (a(y)x + b(y), g(y) ), \; a,b \in \C(y),\; a \neq 0, \; g \in \Sym_4 \} \subseteq \JJ \]
has derived length $5$.
\end{lemma}

\begin{proof}
Define the following elements of $\PGL_2(\C)$:
\[ g_1 (y) = iy, \quad g_2(y) = \frac{y+1}{y-1}, \quad g_3 (y) = \frac{\phantom{-}iy+1}{-iy+1}, \quad g_4(y) = - \frac{1}{y}, \quad g_5(y) = \frac{1}{y}.\]
The following computations show that $g_5 \in D^2( \Sym_4 )$:
\[ [g_1,g_2] = g_3, \quad  [(g_1)^2,g_2] = g_4, \quad [g_3,g_4] =g_5.\]
Setting $G:= \Aff_1 (\C (y) ) \rtimes \Sym_4$, we have shown that $(x,y^{-1}) \in D^2G$. For any $a(y) \in \C(y)^*$, $b(y) \in \C(y)$, the following identities are satisfied:
\begin{equation}[ (a(y)x, y) , (x,y^{-1} ) ] = ( a(y) a(y^{-1})^{-1}x, y ) \quad \text{and} \label{equation: with a} \end{equation}
\begin{equation}[ (x+b(y) , y) , (x,y^{-1} ) ] = ( x+b(y) -b(y^{-1}), y ).\label{equation: with b} \end{equation}
Equation \eqref{equation: with a} shows that if $k \in \{0,1,2\}$ and $(a(y)x,y) \in D^k(G)$, then we have
\[ ( a(y) a(y^{-1})^{-1}x, y ) \in D^{k+1} (G).\]
Hence, we get $( a^4(y) a(y^{-1})^{-4}x, y ) \in D^{3} (G)$ for any $a \in \C(y)^*$. Taking $a = y+1$, we obtain $(y^4x,y) \in D^{3} (G)$.

Analogously, equation \eqref{equation: with b} shows that if $k \in \{0,1,2\}$ and $(x + b(y) ,y) \in D^k(G)$, then $(x + b(y) -b(y^{-1}), y ) \in D^{k+1} (G)$. Hence, we get $( x + 4 b(y) -4 b(y^{-1}) , y ) \in D^{3} (G)$ for any $b \in \C(y)$. Taking $b= \frac{y}{4}$, we obtain $(x +y-y^{-1} ,y) \in D^{3} (G)$.

Finally, since the elements $(y^4x,y)$ and $(x +y-y^{-1} ,y)$ of $D^3(G)$ do not commute, their commutator defines a non-trivial element of $D^4(G)$. This shows that $\length (G) \geq 5$. Since $\ell(\Aff_1(\C(y)))=2$ and $\ell(\Sym_4)=3$, we have
$\ell(G)\leq 5$, and therefore $\ell(G)=5$.
\end{proof}

\begin{proposition} \label{proposition: psi(Jonq)=5}
We have $\psi ( \JJ) = 5$.
\end{proposition}

\begin{proof}
Let $H \subseteq \JJ$ be a solvable subgroup. The short exact sequence
\[ 1 \to \PGL_2 ( \C (y) ) \to \JJ \xrightarrow{\pr_2} \PGL_2 (\C ) \to 1 \]
induces the following short exact sequence
\[ 1 \to H_1 \longrightarrow H \xrightarrow{(\pr_2)_{| H}} H_2 \to 1, \]
where $H_1:= H \cap \PGL_2 ( \C (y) )$ and $H_2 := \pr_2 (H)$. This yields
\[ \length (H) \leq \length(H_1 ) + \length ( H_2 ) \leq 2 \,  \psi ( \PGL_2(\C) ) = 6. \]
By Lemma~\ref{lemma: psi(PGL_2(C))=3} and Remark~\ref{remark: C(y) embedding in C}, the case of equality implies that both $H_1$ and $H_2$ are isomorphic to $\Sym_4$. Hence, $H$ is finite (of cardinality $| \Sym_4 |^2 = (4!)^2$). Since $H$ is finite and solvable, Theorem~\ref{theorem: psi_{finite}(Bir(P2)) =4} gives $\ell(H)\leq4$, a contradiction. Hence, we have proven that $\length (H) \leq 5$ and $\psi ( \JJ ) \leq 5$. Finally, Lemma~\ref{lemma: a subgroup of J whose derived length is 5} shows $\psi ( \JJ ) \geq 5$ and the equality follows.
\end{proof}

\subsection{Halphen surfaces} 

For Halphen surfaces, the required bound follows directly from the natural short exact sequence associated with the action of the automorphism group on the base of the Halphen fibration.

\begin{proposition} \label{proposition: psi(Aut(Halphen)) leq 5}
If $Y$ is a Halphen surface, we have $\psi(\Aut(Y))\leq 5$.
\end{proposition}

\begin{proof}
The action on the base of the Halphen fibration gives a short exact sequence
\[
1\longrightarrow G\longrightarrow\Aut(Y)
\longrightarrow H\longrightarrow1,
\]
where $H\subseteq\PGL_2(\C)$.
The kernel $G$ acts faithfully on the geometric generic fibre $E$,
which is a smooth curve of genus one over $\overline{\C(t)}$.
After choosing an origin on $E$, its automorphism group is the
semidirect product of its abelian translation subgroup and the
origin-preserving automorphism group, which is finite cyclic. Thus $G$ is metabelian. By Lemma~\ref{lemma: psi(PGL_2(C))=3}, we obtain
\[ \psi(\Aut(Y)) \leq \psi(G) + \psi(H) \leq 2+\psi(\PGL_2(\C)) =5. \qedhere \]
\end{proof}

\section{Proof of the main theorem}
\label{section: proof of the main theorem}
We now complete the proof of the main theorem.

\begin{proof}[Proof of Theorem \ref{theorem: main}]
Since we already know that we have $\psi(\bp) \geq 5$, it remains to prove the reverse inequality. Let $H \subseteq \bp$ be a solvable subgroup. We want to prove that $\length(H) \leq 5$.

1) If $H$ contains a loxodromic element, we conclude by Theorem~\ref{theorem: loxodromic case}.

2) If $H$ does not contain a loxodromic element but does contain a parabolic element, then $H$ is conjugate to a subgroup of $\Jonq_{p_0}$ or to a subgroup of $\Aut(Y)$, where $Y$ is a Halphen surface, by \cite[Lemma~21.19]{LamyCremonaBook} (see also \cite[Lemma~2.5]{Urech2021}, where it is explained that this result follows from \cite{Cantat2011}). If $H$ is contained in $\Jonq_{p_0}$ (resp.\ $\Aut(Y)$ for a Halphen surface $Y$), then we have $\ell(H)\leq 5$ by Proposition~\ref{proposition: psi(Jonq)=5} (resp.\ by Proposition~\ref{proposition: psi(Aut(Halphen)) leq 5}).

3) We may therefore assume that every element of $H$ is elliptic. By \cite[Theorem~1.4]{Urech2021}, one of the following three alternatives holds:
\[ \begin{array}{ll}
\text{(i)} & H \text{ preserves a rational fibration;}\\
\text{(ii)} & H \text{ is bounded;}\\
\text{(iii)} & H \text{ is a torsion group.}
\end{array} \]
If $H$ preserves a rational fibration, then, after conjugation,
$H\subseteq\Jonq_{p_0}$, and we conclude by
Proposition~\ref{proposition: psi(Jonq)=5}. If $H$ is bounded, we conclude by Proposition~\ref{proposition: psi_{bounded}(Bir(P2)) =5}.  If $H$ is a torsion group, then $H$ is isomorphic to a bounded subgroup $K\subseteq\bp$ by \cite[Theorem~1.7]{Urech2021}, and we again conclude by Proposition~\ref{proposition: psi_{bounded}(Bir(P2)) =5}.\end{proof}

\end{document}